\documentclass[11pt]{amsart}
\usepackage{amssymb, amsmath, amscd, amsthm, setspace, latexsym, vruler, enumerate}
\usepackage[toc,page]{appendix}
\title[Smallest number of generators and the probability of generating an algebra]{On the smallest number of generators and the
probability of generating an algebra}
\author[R.V.~Kravchenko]{Rostyslav V.~Kravchenko}
\author{Marcin Mazur}
\author[B.V.~Petrenko]{Bogdan V.~Petrenko}

\address{Department of Mathematics\\
Texas A\&M University \\ College Station, TX 77843-3368, USA}
\email{rkchenko@gmail.com}

\address{
Department of Mathematics \\
Binghamton University \\
P.O. Box 6000 \\
Binghamton, NY 13892-6000, USA } \email{
mazur@math.binghamton.edu}

\address{
Department of Mathematics \\ SUNY Brockport \\ 350 New Campus
Drive \\ Brockport, NY 14420, USA
 }

\email{ bpetrenk@brockport.edu}

\newtheorem{theorem}{Theorem}[section]
\newtheorem{lemma}[theorem]{Lemma}
\newtheorem{proposition}[theorem]{Proposition}
\newtheorem{corollary}[theorem]{Corollary}
\newtheorem{remark}[theorem]{Remark}

\newtheorem{conjecture}[theorem]{Conjecture}
\newtheorem{definition}[theorem]{Definition}
\newtheorem{example}[theorem]{Example}

\newtheorem{question}[theorem]{Question}

\def\gen{\text{\rm gen}}
\def\tr{\text{\rm tr}}

\def\PGL{\text{\rm PGL}}
\def\End{\text{\rm End}}
\def\Aut{\text{\rm Aut}}
\def\G{\text{\rm G}}
\def\M{\text{\rm M}}

\def\Gen{\text{\rm Gen}}
\def\den{\text{\rm den}}
\def\g{\text{\rm g}}

\def\ng{\text{\rm ng}}
\def\dd{\displaystyle}
\def\N{\text{\rm N}}
\def\d{\text{\rm d}}

\newcommand{\Dg}{\mathcal D}
\newcommand{\Ag}{\mathcal A}
\newcommand{\Fg}{\mathbb{F}}

\newcommand{\mt}[1]{\M_{#1}(\Fg_q)}
\newcommand{\mx}[2]{\M_{#1}(\Fg_{q^{#2}})}
\newcommand{\fg}{\mathcal F}
\def\sl{\text{\rm sl}}

\def\tr{\text{tr}}
\def\conj{\text{conj}}

\newtheorem{nle}[theorem]{Lemma}
\newtheorem{te}[theorem]{Theorem}
\newtheorem{cl}[theorem]{Corollary}

\newtheorem{pro}[theorem]{Proposition}

\newcommand{\mo}[1]{{ \hbox{\rm{ (mod\ $#1$) }}}}

\newcommand{\f}[1]{\frak{#1}}
\newtheorem*{te*}{Theorem}

\newcommand{\spec}[0]{\text{\rm Spec~}}
\newcommand{\jspec}[0]{\text{\rm ~j-Spec~}}
\newcommand{\mspec}[0]{\text{\rm ~m-Spec~}}

\def\dsp{\def\baselinestretch{1.37}\large}

\begin{document}
\maketitle
%\setvruler[5.6mm][1][1][1][0][150mm][150mm][2.5mm]
\dsp

\begin{abstract}
In this paper we study algebraic and asymptotic properties of generating sets of
algebras over orders in number fields. Let $A$ be an associative algebra over an
order $R$ in an algebraic number field. We assume that
$A$ is a free $R$-module of finite rank. We develop a technique to compute the smallest
number of generators of $A$. For example, we prove that the ring $\M_3(\mathbb{Z})^{k}$ admits two generators
if and only if $k\leq 768$. For a given positive integer $m$, we define the density of the set of all ordered
$m$-tuples of elements of $A$ which generate it as an $R$-algebra.
We express this density as a certain infinite product over the maximal ideals of $R$, and we
interpret the resulting formula
probabilistically. For example, we show that the probability that $2$ random $3\times 3$ matrices
generate the ring $\M_3(\mathbb{Z})$
is equal to $\dd \left( \zeta(2)^2 \zeta(3) \right)^{-1}$,  where $\zeta$ is the
Riemann zeta-function.

\vspace{4mm}

\vspace{3mm}
\end{abstract}

\tableofcontents

\section{Introduction}
Let $R$ be a commutative ring with $1$. Recall that a set
$S$ generates an associative unital $R$-algebra $A$ if the set of all
monomials in the elements of $S$ (including the degree zero monomial $1$)
spans $A$ as an $R$-module. This paper
lays foundation for our program to investigate properties of the sets
of generators of $R$-algebras $A$ whose additive group is a
finitely generated $R$-module. A substantial part of our results grew out of the following question:
given a ring $A$ whose additive group is a free abelian group of finite rank and a positive
integer $k$, what is the probability that $k$ random elements of $A$ generate it as a $\mathbb Z$-algebra?
We will show that this question can be stated in a rigorous way and that it has a very interesting answer.
The following formulas, in which $\zeta$ denotes the Riemann zeta-function, are special cases of our results
(see Theorem~\ref{minnumbergen}):

\begin{itemize}
\item The probability that $m$ random $2\times 2$ matrices generate the ring $\M_2(\mathbb{Z})$
is equal to $\displaystyle \frac{1}{\zeta(m-1) \zeta(m)}$.

\item The probability that $2$ random $3\times 3$ matrices generate the ring $\M_3(\mathbb{Z})$
is equal to $\displaystyle \frac{1}{\zeta(2)^2 \zeta(3)}$.

\item The probability that $3$ random $3\times 3$ matrices generate the ring $\M_3(\mathbb{Z})$
is equal to $\displaystyle \frac{1}{\zeta(2) \zeta(3)\zeta(4)}\prod_p\left(1+\frac{1}{p^2}+
\frac{1}{p^3}-\frac{1}{p^5}\right)$, where the product is taken over all prime numbers.

\end{itemize}

Our main results are obtained for algebras $A$ over an order $R$ in some
number field such that $A$ is a free $R$-module of finite rank. It is not hard though to
extend the results to the case when $R$ is an order in a global field of positive
characteristic (we will address this in a follow-up paper).
 Roughly speaking, a choice of an integral
basis of $R$ and of a basis of $A$ over $R$ allows us to introduce integral coordinates on all cartesian
powers $A^k$, $k\in \mathbb N$. For any
subset $S$ of $A^k$ and any $N$ we consider the finite set $S(N)$ of all points whose coordinates are in
the interval $[-N,N]$. We define the density $\den(S)$ of $S$ as the limit
$\displaystyle \lim_{N\to\infty} \frac{|S(N)|}{|A^k(N)|}$ (we do not claim that it always exists).
Our goal is to calculate the density of the
set of generators of $A$.

\begin{definition}{\rm
Let $A$ be an algebra over a commutative ring $R$, and let $k$ be a positive
integer. We define the set $\Gen_{k}(A,R)$ as follows:
\[\Gen_{k}(A,R) = \{(a_1, \ldots, a_k) \in A^k: a_1, \ldots, a_k~\text{generate}~
A~\text{as an $R$-algebra} \}.\] }
\end{definition}

For the rest of the introduction, we assume that $R$ is an order in a number field $K$ and $A$ is an $R$-algebra
which is free of finite rank $m$ as an $R$-module
(unless stated otherwise).

In Theorem \ref{gendensity} we prove that the set $\Gen_{k}(A,R)$ has density, which we denote by $\den_k(A)$,
and that it can be computed locally as follows:
\begin{equation}\label{introlocal}
\den_{k}(A)=\prod_{ \f p \in \mspec R}\frac{|\Gen_{k}(A/\f pA,R/\f p)|}{
|R/\f p|^{mk}},
\end{equation}
where $\mspec R$ denotes the set of all maximal ideals of $R$.
In order to prove Theorem \ref{gendensity} we had to extend this local-to-global formula for
density to a substantially larger class of sets. This led us to Theorem~\ref{main}, which is of independent
interest and has potential applications to various other questions. Theorem~\ref{main} deals with a finite set
$f_1,\ldots,f_s$ of polynomials in $R[x_1,\ldots, x_n]$ and the set $S$ of all $a\in R^n$ such that the
ideal generated by $f_1(a),\ldots, f_s(a)$ is $R$. It asserts that the set $S$ has density $\den(S)$
and it is given by the following formula:
\[\den(S)= \prod_{ \f p\in \mspec R}\left(1-\frac{t_{\f p}}{|R/\f p|^n}\right),\]
where $t_{\f p}$ is the number of common zeros in $(R/\f p)^n$ of the polynomials $f_1\ldots,f_s$
considered  as polynomials over the field $R/\f p$.

As a first application of our results we answer in section \ref{kap} the following question posed by Ilya Kapovich:
what is the probability that $m$ random elements of a free abelian group of rank $n\leq m$
generate the group? Our results provide a rigorous proof
to the following answer: the probability in question is equal to
$\left( \prod_{k=m-n+1}^{m} \zeta(k)  \right)^{-1}$, where $\zeta$ is the Riemann
zeta-function (when $m=n$ this product should be interpreted as $0$).

In Section 5 we show how (\ref{introlocal}) can be used to get information about the smallest number
of generators of an $R$-algebra $A$.

\begin{definition} {\rm
Let $A$ be a finitely generated $R$-algebra. By $r(A,R)$
we denote the smallest number of generators of $A$ as an $R$-algebra.}
\end{definition}

In Theorem~\ref{denslenstra} we prove that if $k$ is an integer such that
$k\geq r_0:=r(A\otimes_R K,K)+1$ and $k\geq r_{\f p}:=r(A/\f pA, R/\f p R)$ for every maximal ideal $\f p$
of $R$ then $\den_k(A)>0$. Let $r_{f}$ be the largest among the numbers~$r_{\f p}$. Clearly, if $
\den_k(A)>0$ then $A$ can be generated by $k$ elements. Using this remark and Theorem~\ref{denslenstra}
we show in Theorem \ref{lenstra} that the smallest number of generators of $A$ coincides with $r_f$
if $r_f>r_0$ and it is either $r_0$ or $r_0+1$ otherwise. A special case of this result, when $R=\mathbb Z$,
was kindly communicated to us by H.W.~Lenstra \cite{lenstra}. Note that when $r_f=r_0$, we only know that
$r$ is either $r_0$ or $r_0+1$. Nevertheless, it is often possible to prove that $\den_{r_0}(A)>0$
and conclude that $r=r_0$. For example, we have been unable for a long time to find
the largest integer $n$ such that the product $\M_3(\mathbb Z)^n$
of $n$ copies of the matrix ring $\M_3(\mathbb Z)$ admits two
generators as a $\mathbb Z$-algebra. We knew that $n\leq 768$, but
any attempts to construct explicitly two generators for such large
values of $n$ have been beyond our computational ability. It turns
out though that we can prove that $\den_{2}(\M_3(\mathbb
Z)^{768})>0$, hence we get a (non-constructive) proof that
$n=768$ (see Theorem~\ref{3by3}).

In Theorem \ref{swan} we extend Lenstra's
original approach to obtain a similar formula for the smallest number of generators of algebras over
any commutative ring $R$ of dimension at most 1. This formula is reminiscent of the Forster-Swan Theorem
on the number of generators of modules over Noetherian commutative rings (\cite[Theorem 5.8]{matsumura}). By analogy with this result,
in Conjecture \ref{swan1} we propose an extension of our formula to algebras over
general Noetherian rings.

In order to use formula (\ref{introlocal}) in concrete cases one needs to be able to compute
the numbers $|\Gen_{k}(A/\f pA,R/\f p)|$. This leads us to results of Sections~\ref{s6} and \ref{s7}, where we
study these numbers under the assumption that $A/\f p A$ is a product of matrix algebras.
After various reductions in Section~\ref{s6} we derive explicit formulas for
$|\Gen_{k}(\M_n(\mathbb F),\mathbb F)|$,
where $\mathbb F$ is a finite field and $n=2,3$. Furthermore, we get a lower bound when $n>3$
(Proposition~\ref{lowerbound}  ). As a corollary, we prove that the probability that $m$ matrices
in $\M_n \left( \mathbb{F}_q \right)$, chosen under the uniform
distribution, generate the $\mathbb{F}_q$-algebra $\M_n \left( \mathbb{F}_q \right)$
tends to $1$ as $q+m+n \to \infty$ (see Corollary~\ref{infinity}).
This result proves and vastly generalizes the conjectural formula (17)
on. p.~27 of \cite{pet-sid}.
The case of $n=2$ and some of the results of Section 6 have been discussed earlier in an unpublished preprint
\cite{kp}, which was the starting point for the present work. This part of our paper  has been influenced
by ideas of Philip Hall \cite{hall}.

In Section~\ref{numfield} the results of Sections~\ref{s6} and \ref{s7} are applied to finite products
of matrix algebras over the ring of integers in a number field.

To state some of our remaining results, we need the following
notation.

\begin{definition}\label{defgen} {\rm Let $m,n \ge 1$ be integers and let $A$ be an $R$-algebra. We introduce the
following notation:
\begin{enumerate}[\rm (i)]
\item $\gen_m(A,R)$ is the largest $k\in \mathbb Z\cup\{\infty\}$ such that $r(A^k,R)\leq m$.

\item  $\gen_{m,n}(q)=  \gen_m \left(\M_n(\mathbb{F}_q),\mathbb{F}_q \right)$.

\item $\g_{m,n}(q)=|\Gen_{m}\left((\M_n(\mathbb F_q), \mathbb{F}_q  \right)|$.
\end{enumerate}}
\end{definition}

We show in Proposition~\ref{orbits} that
\[
\dd \gen_{m,n}(q) =\frac{\g_{m,n}(q)}{|\PGL_n(\mathbb{F}_q)|}\]

and $r(\M_n(\mathbb F_q)^{1+\gen_{m,n}(q)}, \mathbb F_q) = m+1 $ by Corollary~\ref{jump}.

Here are some special cases of our results in Section~\ref{numfield}.

{\bf{1.}} $\dd \gen_{m,2}(q) = \frac{{q}^{2\,m-1}\,\left( {q}^{m}-1\right) \,\left(
{q}^{m}-q\right)}{q^2-1}$.

{\bf{2.}} $\dd \gen_m(\M_2(\mathbb{Z}), \mathbb{Z}) = \gen_{m,2}(2) =
\frac{{2}^{2\,m-1} \, \left( {2}^{m}-2\right) \,\left(
{2}^{m}-1\right)}{3}$.

{\bf{3.}} $\dd \gen_{m,3}(q)=\frac{{q}^{3\,m-3}\,\left( {q}^{m}-1\right) \,\left(
{q}^{m}-q\right) \,\left( {q}^{m}+q\right) }{{\left( q-1\right) }^{2}\,\left(
q+1\right) \,\left( {q}^{2}+q+1\right) } \, \times $
$$\left( {q}^{3\,m}-{q}^{m+2}+{q}^{2\,m}-2\,{q}^{m+1}-{q}^{m}+{q}^{3}+{q}^{2}\right).$$

{\bf{4.}} $\dd \gen_m(\M_3(\mathbb{Z}), \mathbb{Z}) = \gen_{m,3}(2)=$
$$ \frac{\left( {2}^{m}-2\right) \,\left( {2}^{m}-1\right) \,\left( {2}^{m}+2\right)
\,\left( {2}^{3\,m}+{2}^{2\,m}-{2}^{m+3}-{2}^{m}+12\right) \,{2}^{3\,m-3}}{21}.$$

\vspace{3mm}
The techniques developed so far can be applied to any finitely generated $\mathbb Z$-algebra
whose reduction modulo every prime is a direct sum of matrix rings over finite fields.
However, among maximal orders in semi-simple algebras over $\mathbb Q$ the only such
algebras are the maximal orders in matrix rings by the
Hasse-Brauer-Noether-Albert Theorem. In order to extend our results to maximal orders in other semi-simple algebras
we need to obtain formulas for the number of generators of algebras over finite field which have
non-trivial Jacobson radical. This will be done in a subsequent paper.
Let us just mention here a special case,
when $A$ is a maximal order in the quaternion algebra $\mathbb Q(i,j)$ ($i^2=-1=j^2$).
For any odd prime $p$ we have $A/pA\cong \M_2(\mathbb F_p)$, so $A$ and $\M_2(\mathbb Z)$
differ only at the prime $2$ and at infinity. Note that $A/2A$ is a commutative algebra over $\mathbb F_2$
whose quotient modulo the Jacobson radical is the field $\mathbb F_4$. Since $\mathbb F_4^{16}$ cannot
be generated by $2$ elements, we see that $A^{16}$ requires at least three generators. It can be verified that
$A^{15}$ admits two generators. So $A$ can be distinguished from $\M_2(\mathbb Z)$ by counting
the smallest number of generators of powers of these two algebras. Note that for the integral quaternions
$\mathbb Z[i,j]$
already $\mathbb Z[i,j]^4$ requires at least three generators. In a subsequent paper we will extend this
observations to a much larger class of orders.

In another work in progress we apply the techniques developed in the present paper to study generators
of various non-associative algebras. Our technique applies to any finitely generated $R$-module equipped
with an $R$-bilinear form, but we focus mainly on Lie algebras and Jordan algebras. For example,
we show that the probability that $m$ random elements generate the Lie ring $\sl_2(\mathbb{Z})$ of $ 2 \times 2$
integer matrices with zero trace is equal to $\displaystyle \frac{1}{\zeta(m-1) \zeta(m)}$.

\paragraph*{\bf{Acknowledgments}}
It is our pleasure to thank Max Alekseyev, Nigel Boston, Evgeny Gordon, Rostislav Grigorchuk,  Ilya Kapovich,
Martin Kassabov, Hendrik Lenstra, Pieter Moree,
Tsvetomira Radeva, Peter Sarnak, Said Sidki, John Tate, and Paula Tretkoff.
The first author was supported by NSF
Grant DMS-0456185.
The
third author thanks the Max Planck
Institute for Mathematics for the warm hospitality, unique research opportunities, and financial support
during his visit
in July-August of 2009.

\section{Preliminary results}
%Let $R$ be a commutative ring. Unless stated otherwise, all $R$-algebras are assumed to be finite over $R$
%(i.e. are finitely generated as $R$-modules).
{\em Let $R$ be a commutative ring with $1$. Unless stated otherwise,
all $R$-algebras are assumed to be associative, unital, and finitely generated as an $R$-module.}

In this section we collect several fairly straightforward observations which are used through the paper.
Let $A$ be an $R$-algebra. Recall that elements $a_1,\ldots,a_k$ generate $A$ as an $R$-algebra
if all the (non-commutative) monomials in $a_1,\ldots,a_k$, including the degree zero monomial $1$,
generate $A$ as an $R$-module. We say that $a_1,\ldots,a_k$ {\bf strongly generate} $A$ as an $R$-algebra
if already all the (non-commutative) monomials in $a_1,\ldots,a_k$ of positive degree generate $A$ as
an $R$-module.

\begin{nle}\label{*l1}
Suppose that there does not exist an $R$-algebra homomorphism $A\longrightarrow R/I$
%Suppose that there are no $R$-algebra homomorphisms $A\longrightarrow R/I$
for any proper ideal $I$ of $R$.
Then any set which generates $A$ as an $R$-algebra also strongly generates $A$.
\end{nle}
\begin{proof}
Suppose that $a_1,\ldots,a_k$ generate $A$ as an $R$ algebra and let $J$ be the $R$ submodule of $A$ generated
by all the (non-commutative) monomials in $a_1,\ldots,a_k$ of positive degree. Then $R\cdot 1+J=A$.
Since $J$ is closed under multiplication, it is a two sided ideal of $A$ and $A/J\cong R/(R\cap J)$.
By our assumption, $R\cap J$ cannot be a proper ideal of $R$, so $R\cdot 1 \subset J$ and $J=A$.
\end{proof}

\begin{example}\label{unity} {\rm Let the algebra $A = \prod_{i=1}^n \M_{m_k}(R)$ be a finite product of matrix algebras
over $R$, with each $m_i \ge 2$. Then any set which generates $A$ as an $R$-algebra also strongly generates $A$. This is a
direct consequence of Lemma \ref{*l1} and the remark that $A$ has no non-trivial commutative quotients.}
\end{example}

In this paper we decided to focus on unital algebras and we do not discuss strong generators.
However most of our results can be easily modified to sets of strong generators
and algebras which are not necessarily unital. One can also reduce questions about strong generators
to generators using the following observation. Recall that if $A$ is an $R$-algebra (unital or not)
we can construct a unital algebra $A^{(1)}$ which is $R\oplus A$ as an $R$-module with multiplication
defined by $(r,a)(s,b)=(ab+rb+sa,rs)$. We have the following lemma.

\begin{lemma}\label{add1}
Let $a_1,\ldots,a_k\in A$. Then the following conditions are equivalent:
\begin{enumerate}[\rm (1)]
\item  $a_1,\ldots,a_k$ strongly generate $A$ as an $R$-algebra.

\item $(r_1,a_1),\ldots,(r_k,a_k)$ generate $A^{(1)}$ as an $R$-algebra for any elements
$r_1,\ldots, r_k\in~R$.

\item  $(r_1,a_1),\ldots,(r_k,a_k)$ generate $A^{(1)}$ as an $R$-algebra for some elements
$r_1,\ldots, r_k\in R$.
\end{enumerate}
\end{lemma}
\begin{proof}
We identify $A$ with the ideal $\{0\}\oplus A$ in $A^{(1)}$. Assume $(1)$ and let $r_1,\ldots,r_k\in R$.
Since $(0,a_i)=(r_i,a_i)-r_i(1,0)$, the $R$-subalgebra $B$
of $A^{(1)}$ generated by $(r_1,a_1),\ldots,(r_k,a_k)$ contains all monomials in $(0,a_1),\ldots, (0,a_k)$,
hence it contains $A$. Since $B$ also contains $R\oplus \{0\}$, we see that $A^{(1)}=B$. Thus
$(1)$ indeed implies $(2)$. The condition $(3)$ is clearly a consequence of $(2)$. Assume $(3)$ and let $C$ be
the subalgebra of $A$ strongly generated by $a_1,\ldots,a_k$. Note that any monomial of positive degree
in $(r_1,a_1),\ldots,(r_k,a_k)$ is of the form $(r,c)$ for some $r\in R$ and $c\in C$. By the assumption in
$(3)$,
for any $a\in A$ there is $r\in R$ such that $(r,a)$ is an $R$-linear combination of
monomials of positive degree in $(r_1,a_1),\ldots,(r_k,a_k)$.
It follows that $a\in C$. Thus $C=A$, which shows that $(1)$ follows from $(3)$.
\end{proof}

The following observation is straightforward.

\begin{lemma}\label{add1mod}
Let $A$ be an $R$-algebra.
For any ideal $I$ of $R$ we have $A^{(1)}/IA^{(1)}=(A/IA)^{(1)}$, where the adjunction of unity on the right
is in the category of $R/I$-algebras.
\end{lemma}

\begin{definition}\label{Gen}{\rm
For an $R$-algebra $A$ and a positive integer $k$ we denote by $\Gen_{k}(A, R)$ the set of all $k$-tuples
$(a_1,\ldots,a_k)\in A^k$ which generate $A$ as an $R$-algebra. When there is no danger of confusion,
we write $\Gen_{k}(A)$ for $\Gen_{k}(A, R)$.}
\end{definition}

%For an $R$-algebra $A$ and a positive integer $k$ we denote by $Gen_{k}(A)$ the set of all $k$-tuples
%$(a_1,\ldots,a_k)\in A^k$ which generate $A$ as an $R$-algebra.

\begin{nle}\label{localgen}
Elements $a_1,\ldots,a_k$ generate $A$ as an $R$-algebra if and only if for every maximal ideal $\f m$ of $R$
the images of $a_1,\ldots,a_k$ in $A\otimes_{R} R/\f m =A/\f m A$ generate $A/\f mA$ as an $R/\f m$-algebra.
\end{nle}
\begin{proof}
Let $J$ be the $R$ submodule of $A$ generated
by all the (non-commutative) monomials in $a_1,\ldots,a_k$. By \cite[Theorem 4.8]{matsumura}, $A=J$
iff $A/J\otimes_{R} R/\f m=0$ for every maximal ideal $\f m$ of $R$. The result follows from the simple
remark that $A/J\otimes_{R} R/\f m=0$ iff the images of $a_1,\ldots,a_k$ in $A/\f mA$ generate it as an $R/\f m$-algebra.
\end{proof}

\begin{nle}\label{fieldspan}
Let $R$ be a field and let $A$ be an $R$-algebra of dimension $m$.
Elements $a_1,\ldots,a_k$ generate $A$ as an $R$-algebra if and only if the (non-commutative) monomials in $a_1,\ldots,a_k$
of degree $< m$ span $A$ as an $R$-vector space.
\end{nle}
\begin{proof}
Let $A_i$ be the subspace of $A$ spanned by all the monomials in $a_1,\ldots,a_k$
of degree $\leq i$. Clearly $A_0\subseteq A_1\subseteq A_2\subseteq \dots$.
We also see that
\[A_{i+1}=A_i+a_1A_{i}+a_2A_{i}+\ldots+a_kA_i \]
for any $i$.
It follows that if $A_i=A_{i+1}$ for some $i$, then
$A_j = A_i$ for all $j\geq i$. Since $\dim_R A_{m} \leq m$, we must have $A_i=A_{i+1}$ for some $i<m$.
Thus $A_i=A_{m-1}$ for all $i\geq m$. This proves that $a_1,\ldots,a_k$ generate
$A$ as an $R$-algebra iff $A=A_{m-1}$.

\end{proof}

\begin{nle}\label{newspan}
Suppose  that $A$ can be generated by $m$ elements as an $R$-module. Elements $a_1,\ldots,a_k$ generate $A$ as an
$R$-algebra if and only if the (non-commutative) monomials in $a_1,\ldots,a_k$
of degree $< m$ generate $A$ as an $R$-module.
\end{nle}
\begin{proof}
Suppose that $a_1,\ldots,a_k$ generate $A$ as an $R$-algebra
Let $A_i$ be the $R$-submodule of $A$ generated by all the monomials in $a_1,\ldots,a_k$
of degree $\leq i$. For any maximal ideal $\f m$ of $R$ the dimension of $A/\f m A$ over $R/\f m$
does not exceed $m$. Thus $A/A_{m-1}\otimes_{R} R/\f m=0$ for every maximal ideal $\f m$ of $R$
by Lemma \ref{fieldspan}. Hence $A=A_{m-1}$ by \cite[Theorem 4.8]{matsumura}.
\end{proof}

Recall that $\spec R$ is the set of all prime ideals of $R$ equipped with the Zariski topology and
$\mspec R$ is the subspace of $\spec R$ consisting of all maximal ideals. For $\f p \in \spec R$ we
denote by $R_{\f p}$ the localization of $R$ at the prime ideal $\f p$ and we set
$A_{\f p}=R_{\f p}\otimes_R A$. The residue field $R_{\f p}/\f p R_{\f p}$ is denoted by $\kappa(\f p)$.
Recall that $\kappa(\f p)$ coincides with the field of fractions of $R/\f p$.

\begin{definition}{\rm
We say that the elements $a_1,\ldots,a_k$ generate $A$ at a prime ideal $\f p$ of $R$ if their images
in $\kappa(\f p)\otimes_R A$ generate $\kappa(\f p)\otimes_R A$ as a $\kappa(\f p)$-algebra. Equivalently,
$a_1,\ldots,a_k$ generate $A$ at $\f p$ if their images in $A_{\f p}$ generate $A_{\f p}$ as an
$R_{\f p}$-algebra.}
\end{definition}

\begin{nle}\label{open}
Let $a_1,\ldots,a_k\in A$. The set of all prime ideals $\f p$ such that $a_1,\ldots,a_k$
generate $A$ at $\f p$ is open.
\end{nle}

\begin{proof}
Let $B$ be the $R$ submodule of $A$ generated by all monomials in $a_1,\ldots,a_k$ of degree $< m$,
where $m$ is such that $A$ can be generated by $m$ elements as an $R$-module. By Lemma \ref{newspan},
the images of $a_1,\ldots,a_k$ in $A_{\f p}$ generate $A_{\f p}$ as an $R_{\f p}$-algebra iff
$(A/B)_{\f p}=0$. Since the support of a finitely generated $R$-module is closed, the result follows.
\end{proof}

\begin{cl}\label{genopen}
For any positive integer $k$ the set
\[ U_k=\{\f p\in \spec R: A_{\f p} \text{can be generated by $k$ elements as an
$R_{\f p}$-algebra}\}
\]
is open.
\end{cl}
\begin{proof}
Suppose that $A_{\f p}$ is generated by $k$ elements as an $R_{\f p}$-algebra.
We may choose  elements $a_1,\ldots,a_k$ in $A$ which generate
$A$ at $\f p$. By Lemma \ref{open}, there is an open neighborhood
of $\f p$ such that $a_1,\ldots,a_k$ generate $A$ at $\f q$ for each
$\f q$ in this neighborhood.  This shows that $U_k$ is open.
\end{proof}

\begin{pro}\label{matrixgen}
Suppose that $A=\prod_{i=1}^s A_i$ is a product of $R$-algebras $A_1,\ldots,A_s$ such that
for any maximal ideal $\f m$ of $R$ and any $i\neq j$ the $R/\f m$-algebras
$A_i\otimes_{R} R/\f m$ and $A_j\otimes_{R} R/\f m$ do not have isomorphic quotients.
Then $\Gen_k(A)=\prod_{i=1}^s \Gen_k(A_i)$ under the natural identifications.

\end{pro}
\begin{proof}
The proposition says that a sequence $a_1,\ldots,a_k$ of elements in
$A$ generates $A$ as an $R$-algebra iff for every $i$ the
projection of these sequence to $A_i$ generates $A_i$ as an
$R$-algebra. The implication to the right is clear.
Since $A\otimes_{R} R/\f m=\prod_{i=1}^s (A_i\otimes_{R} R/\f m)$,
Lemma \ref{localgen} reduces the proof to the case when $R$ is a field.
Suppose that a sequence $a_1,\ldots,a_k$ of elements in $A$
has the property that for every $i$ the
projection of these sequence to $A_i$ generates $A_i$ as an
$R$-algebra. Let $B$ be the $R$-subalgebra of $A$ generated by $a_1,\ldots,a_k$.
By our assumption, the projection $\pi_i:B\longrightarrow A_i$ is surjective.
Let $J_i=\ker \pi_i$. Since $A_i$ and $A_j$ have no isomorphic quotients for $i\neq j$,
we conclude that $J_i+J_j=B$ for $i\neq j$ (for otherwise, $J=J_i+J_j$ would be a proper ideal of $B$
and $B/J$ would be isomorphic to a quotient of $A_i$ and a quotient of $A_j$).
The Chinese Remainder Theorem implies now that $B=A$.

\end{proof}

\begin{example}\label{matrixproduct}
{\rm Let $A_i=\M_{n_i}(R)^{m_i}$ be the product of $m_i$ copies of the $n_i\times n_i$ matrix ring over $R$,
where $n_i\neq n_j$ for $i\neq j$. Then for any maximal ideal $\f m$ of $R$ we have
$A_i\otimes_{R} R/\f m=\M_{n_i}(R/\f m)^{m_i}$. Consider two distinct indexes $i$, $j$. If the $R/\f m$-algebras
$A_i\otimes_{R} R/\f m$ and $A_j\otimes_{R} R/\f m$  had isomorphic quotients, they would have
isomorphic quotients which are simple $R/\f m$-algebras.
Clearly any simple quotient of $\M_{n_i}(R/\f m)^{m_i}$
is isomorphic to $\M_{n_i}(R/\f m)$. Since $\M_{n_i}(R/\f m)$ and $\M_{n_j}(R/\f m)$ are not isomorphic
(they have different dimensions over $R/\f m$), we see that the $R/\f m$-algebras
$A_i\otimes_{R} R/\f m$ and $A_j\otimes_{R} R/\f m$ do not have isomorphic quotients. Therefore the assumptions of
Proposition \ref{matrixgen} are satisfied and
\[\Gen_k\left(\prod_{i=1}^{s}\M_{n_i}(R)^{m_i}\right)=\prod_{i=1}^s \Gen_k\left(\M_{n_i}(R)^{m_i}\right).
\]}
\end{example}

Recall that in Definition~\ref{defgen} we defined $\gen_m(A,R)$ as the largest $k$ such that
$A^k$ admits $m$ generators as an $R$-algebra. The following proposition implies that
if $\gen_m(A,R)$ is finite then $\gen_{m+1}(A,R)>\gen_m(A,R)$.

\begin{proposition}\label{genjump}
Let $A$ be an $R$-algebra and let $n$ be a positive integer. If $A^n$ can be generated by $m$ elements as an
$R$-algebra then $A^{n+1}$ can be generated by $m+1$ elements.
\end{proposition}
\begin{proof}
Let $a_i=(a_{i,1},\ldots,a_{i,n})$, $i=1,\ldots,m$ generate $A^n$.  Let $b_i=(a_{i,1},....,a_{i,n},a_{i,1})$,
$i=1,\ldots,m$
and set $b_{m+1}=(0,\ldots,0,1)$. For any $w=(w_1,\ldots,w_n)\in A^n$ there is a non-commutative polynomial
$p(x_1,\ldots,x_m)$ with coefficients in $R$
such that $w=p(a_1,\ldots, a_m)$. Then $p(b_1,\ldots, b_m)=(w_1,\ldots,w_m,w_1)$.
It follows that $b_{m+1}p(b_1,\ldots, b_m)=(0,\ldots,0,w_1)$
and $p(b_1,\ldots, b_m)-b_{m+1}p(b_1,\ldots, b_m)=(w_1,\ldots,w_m,0)$. Thus the algebra generated by
$b_1,\ldots,b_{m+1}$ coincides with $A^{n+1}$.
\end{proof}

\begin{corollary}\label{jump}
Let $A$ be an $R$-algebra. If $\gen_m(A,R)$ is finite then $r(A^{1+\gen_m(A,R)},R)=m+1$.
\end{corollary}

We end this section with a discussion of an effective method of checking if given elements generate
an $R$-algebra $A$. The key observation is contained in the following simple Lemma:

\begin{nle}\label{monomials}
Let $A$ be an $R$-algebra generated as $R$-module by elements $u_1,\ldots,u_m$ and let $k\geq 1$
be an integer. For every monomial $M=M(x_1, \ldots , x_k)$ in $k$ non-commuting
variables $x_1,\ldots , x_k$ there are polynomials $p_j^M(x_{1,1},\ldots,x_{k,m}) \in
R\left[x_{1,1}, \ldots ,x_{k,m}\right]$, $j=1,\ldots,m$,  such that the degree of each $p_j^M$
does not exceed the degree of $M$ and
\[M(a_1,\ldots,a_k)=\sum_{i=1}^{m}p_i^M(a_{1,1},\ldots,a_{k,m})u_i\]
whenever $a_{i,j}\in R$ satisfy $a_i=\sum_{j=1}^m a_{i,j}u_j$.
\end{nle}

\begin{proof}
There exist elements $c_{i,j,s}\in R$, $1\leq i,j,s\leq m$ such that
$u_iu_j=\sum_{s=1}^m c_{i,j,s}u_s$. Note that these elements are not unique, unless
$A$ is a free $R$-module with basis $u_1,\ldots,u_m$ (this is the case we are mainly interested in).
We fix some choice of elements $c_{i,j,s}$ and call them the structure constants for $A$.
Furthermore, choose and fix $r_i\in R$, $i=1,\ldots,m$ such that $1=\sum r_i u_i$.
We prove the lemma by induction on the degree of $M$. If degree of $M$ is $0$ then $M=1$
and we can choose constant polynomials $p_i^M=r_i$. Suppose that the lemma holds for
all monomials of degree less than $n$ and let $M$ be a monomial of degree $n$.
Then $M=Nx_{t}$ for some monomial $N$ of degree $n-1$ and some $t\in \{1,\ldots,k\}$.
If $a_i=\sum_{j=1}^m a_{i,j}u_j$, with $a_{i,j}\in R$, $1\leq i\leq k$, then
\[M(a_1,\ldots,a_k)=N(a_1,\ldots,a_k)\sum_{j=1}^m a_{t,j}u_j=\]
\[=\left(\sum_{i=1}^{m}p_i^N(a_{1,1},\ldots,a_{m,k})u_i\right)
\left(\sum_{j=1}^m a_{t,j}u_j\right)=\]
\[=\sum_{i=1}^{m}\sum_{j=1}^mp_i^N(a_{1,1},\ldots,a_{m,k})a_{t,j}\sum_{s=1}^m c_{i,j,s}u_s =\]
\[=\sum_{s=1}^m \left(\sum_{i=1}^{m}\sum_{j=1}^mp_i^N(a_{1,1},\ldots,a_{m,k})a_{t,j}c_{i,j,s}\right)u_s.
\]
This proves that the polynomials
\[ p_s^M= \sum_{i=1}^{m}\sum_{j=1}^mc_{i,j,s}p_i^Nx_{t,j}, \ \ s=1,\ldots,m\]
have the required properties.
\end{proof}

\begin{nle}\label{polynomialsgen}
Let $A$ be an $R$-algebra which is a free $R$-module with a basis $u_1,\ldots,u_m$ and let $k\geq 1$
be an integer. There is a finite set $T\subseteq
R\left[x_{1,1}, \ldots ,x_{k,m}\right]$ of polynomials of degree not exceeding $m^2$
such that for any commutative $R$-algebra $S$ the elements $a_i=\sum_{j=1}^m a_{i,j}\otimes u_j$,
$1\leq i\leq k$, of
$S\otimes_R A$, where $a_{i,j}\in S$,
generate $S\otimes_R A$ as an $S$-algebra if and only if the ideal of $S$ generated by all the values
$f(a_{1,1},\ldots ,a_{k,m})$, $f\in T$, coincides with $S$.
\end{nle}

\begin{proof}
Consider polynomials $p_j^M$ described in Lemma \ref{monomials}. It is clear that the same polynomials
(or rather their images in $S\left[x_{1,1}, \ldots ,x_{k,m}\right]$) work for the $S$ algebra
$S\otimes_R A$ and its generators $1\otimes u_1,\ldots, 1\otimes u_m$.
Let ${\mathcal M}={\mathcal M}(x_{i,j})$ be the matrix whose rows are labelled in some way by the monomials $M$
of degree $< m$
in non-commuting variables $x_{1}, \ldots, x_{k}$, and whose row with label
$M$ is
\[
\left(p_1^M \left(x_{1,1}, \ldots ,x_{k,m} \right), \ldots ,p_m^M \left(x_{1,1},\ldots ,x_{k,m} \right) \right).
\]
The $m\times m$ minors of ${\mathcal M}$
are polynomials in $R\left[x_{1,1}, \ldots ,x_{k,m}\right]$ of degree $\leq m^2$. Consider the set $T$ of all
these minors. Consider elements $a_i=\sum_{j=1}^m a_{i,j}\otimes u_j\in S\otimes_R A$, where $a_{i,j}\in S$
and $1\leq i\leq k$.
Let $B$ be the set of all elements of the form $M(a_1,\ldots,a_k)$, where $M$ is a monomial of degree $< m$.
By Lemmas \ref{newspan} and \ref{localgen}, the elements $a_1,\ldots,a_k$ generate
$S\otimes_R A$ as an $S$-algebra if and only if for every maximal ideal $\f m$ of $S$, the image of the set $B$
in $S\otimes_R A/\f m (S\otimes_R A)$ spans the $S/\f m$-vector space $S\otimes_R A/\f m (S\otimes_R A)$. This is
equivalent to saying that the
reduction modulo $\f m$ of the matrix ${\mathcal M}(a_{i,j})$ has rank $m$, which in turn is
equivalent to the condition that at least one of the $m\times m$ minors of ${\mathcal M}(a_{i,j})$
does not belong to $\f m$. Thus the set $T$ of all the $m\times m$ minors of ${\mathcal M}(x_{i,j})$
has the required property.
\end{proof}

\section{The density of the set of ordered $k$-tuples which generate an algebra}\label{s3}
The results of this section arose from our attempt to answer the following question: what is the probability
that $k$ random elements of a ring $A$, whose additive group is free of finite rank, generate $A$ as a ring.
Before we answer this question, we need to make it more precise.
We will discuss it in a slightly more general context.

Throughout this section $K$ will be a number field of degree $d$ over $\mathbb Q$,
with the ring of integers $O_K$. We work with an order $R$ in $K$, i.e. $R$ is a subring of $K$
which is free of rank $d$ as a $\mathbb Z$-module. We fix an integral basis $w_1,\ldots,w_d$ of
$R$ over $\mathbb Z$. Any element $r$ of $R$ can be uniquely
written as $r=\sum r_i w_i$ with $r_i\in \mathbb Z$. For a
positive integer $N$ we denote by $R(N)$ the set of all $r\in R$
such that $|r_i|\leq N$ for all $i$. Clearly $|R(N)|=(2N+1)^d$.

Let $A$ be an $R$-algebra which is free of finite rank $m$ as an $R$-module. Fix a basis $e_1,\ldots,e_m$ of $A$
over $R$.
This choice allows us to identify $A$ and $R^m$. Using this identification we define $A(N)$ as $R^m(N)$, so
$|A(N)|=(2N+1)^{dm}$.
We define the density $\den_k(A)$ of the set of $k$ generators of $A$
as an $R$-algebra as follows:
\begin{definition}{\rm
\[ \den_k(A)=\lim_{N\to\infty}\frac{|\Gen_k(A)\cap A(N)^k|}{(2N+1)^{dmk}}.
\]}
\end{definition}
At the moment it is not clear whether the limit on the right of the last formula exists. We will show, however,
that it exists and it is independent of the choice of an integral basis of $R$ and the choice of a basis
of $A$ over $R$.

Consider a maximal ideal $\f p$ of $R$. We denote by $\mathbb F_{\f p}$ the field $R/\f p$ and by $\N(\f p)$
its cardinality.
%Let $t_{\f p}$ be the cardinality of the complement of the set
%$\Gen_k(A\otimes_R \mathbb F_{\f p})$ in $(A\otimes_R \mathbb F_{\f p})^k$. In other words, $t_{\f p}$
%is the number of $k$-tuples of elements of $A\otimes_R \mathbb F_{\f p}$ which do not generate
%$A\otimes_R \mathbb F_{\f p}$ as an $\mathbb F_{\f p}$-algebra.
Recall that we say that elements $a_1,\ldots, a_k$ of $A$
generate $A$ at $\f p$ if their images in $A\otimes_R \mathbb F_{\f p}$ generate $A\otimes_R \mathbb F_{\f p}$
as an $\mathbb F_{\f p}$-algebra. Let $g_k(\f p, A)$ be the cardinality of the set
$\Gen_k(A\otimes_R \mathbb F_{\f p})$. In other words, $g_k(\f p, A)$
is the number of $k$-tuples of elements of $A\otimes_R \mathbb F_{\f p}$ which generate
$A\otimes_R \mathbb F_{\f p}$ as an $\mathbb F_{\f p}$-algebra.
It is not hard to see that the density of the set
$\Gen_k(\f p, A)$ of
all $k$-tuples in $A^k$ which generate $A$ at $\f p$ is
\[ \lim_{N\to\infty}\frac{|\Gen_k(\f p,A)\cap A(N)^k|}{(2N+1)^{dmk}}=\frac{g_k(\f p, A)}{\\N(\f p)^{mk}}.\]
Note that by Lemma \ref{localgen}, a given $k$-tuple of elements of $A$ generates it as an $R$-algebra iff
it generates $A$ at $\f p$ for every maximal ideal $\f p$ of $R$. Suppose now that the events
``generate at $\f p$" are independent for different maximal ideals (we use this notion in a very intuitive
sense here, just to motivate our result). It would mean that the probability that random $k$ elements of $A$
generate it as an $R$-algebra is the product of the numbers $g_k(\f p, A)/\N(\f p)^{mk}$ over all maximal
ideals $\f p$ of $R$. One of the main results of this section is a rigorous proof that this is indeed true.
In other words, we prove the following theorem

\begin{te}\label{gendensity}
Let $A$ be an $R$-algebra which is free of rank $m$ as an $R$-module and let $k>0$ be an integer.
For a maximal ideal $\f p$ of $R$ denote by $g_k(\f p, A)$ the number of $k$-tuples of elements
of $A\otimes_R \mathbb F_{\f p}$ which generate
$A\otimes_R \mathbb F_{\f p}$ as an $\mathbb F_{\f p}$-algebra. Then
\begin{equation}\label{dens1}
 \den_k(A)=\prod_{\f p\in \mspec R}\frac{g_k(\f p, A)}{\N(\f p)^{mk}}.
\end{equation}
\end{te}
Note that this result establishes in particular the existence and independence of all the choices
of the limit defining the quantity $\den_k(A)$.

We will derive Theorem \ref{gendensity} as a consequence of a more general result. To this end
consider the set $T=\{f_1,\ldots,f_s\}$ of polynomials in $R[x_{1,1},\ldots,x_{k,m}]$ established in
Lemma \ref{polynomialsgen}
(under our choice of a basis of $A$ over $R$). We identify $A^k$ with the set $R^{mk}$ so that
a tuple $(a_1,\ldots,a_k)\in A^k$ corresponds to $(a_{i,j})\in R^{mk}$ iff $a_i=\sum_{j=1}^{m}a_{i,j}e_j$.
Note that according to Lemma \ref{polynomialsgen}, the element $a=(a_{i,j})\in R^{mk}$ corresponds to a
$k$-tuple in $\Gen_k(A)$ iff the ideal of $R$ generated by the elements $f_1(a),\ldots,f_s(a)$ coincides
with $R$. Moreover, $a$ corresponds to a $k$-tuple which generates $A$ at $\f p$ iff $f_i(a)\not\in \f p$
for some $i$. It follows that $g_k(\f p, A)=\N(\f p)^{mk}-t_{\f p}$, where $t_{\f p}$
is the number of solutions to $f_1=\ldots=f_s=0$ over the finite field $\mathbb F_{\f p}$. It is clear now
that Theorem \ref{gendensity} is a consequence of the following result.

\begin{te}\label{main} Let $R$ be an order in a number field $K$
and  let $T=\{f_1,\ldots,f_s\}\subset R[x_1,\ldots,x_n]$ be a finite set
of polynomials. Define
\[ S=S(T)=\{ x=(x_1,\ldots,x_n)\in R^n: \ \text{the ideal generated by}\  \]
\[f_1(x),\ldots,f_s(x) \ \text{is}\  R\}.\]
For each maximal ideal $\f p$ of R let $t_{\f p}$ be the number of
solutions to $f_1=\ldots=f_s=0$ over the finite field $\mathbb F_{\f
p}=R/\f p$. For a positive integer $N$
let $S_N=S_N(T)=\{x\in S:x_i\in R(N),~i=1,2,\ldots,n\}$. Then
\begin{equation}\label{denformula}
 \lim_{N\to\infty} \frac{|S_N|}{(2N+1)^{dn}}=\prod_{\f p \in \mspec R} \left(1-\frac{t_{\f p}}{\N(\f p)^n}\right).
\end{equation}
\end{te}

A proof of Theorem~\ref{main} is given in the next section. Note that for $s=2$, $R=\mathbb Z$,
and polynomials $f_1,f_2$ which do not have a non-constant common factor this result was
proved by Poonen \cite{poonen} in a slightly more general form (in Poonen's paper
limit is taken over boxes whose all sides increase to infinity; here we only deal with boxes which are cubes).
Poonen's result was inspired by an earlier paper by T. Ekedahl \cite{chinese}, where a similar result
has been established. In a recent paper \cite{arnold} Arnold
considers the set of pairs of relatively prime integers as a subset of $\mathbb Z^2$ and proves that its density
can be computed by using sets of the form $nG$, where $G$ is any polygon (so our case corresponds
to $G$ being the square $|x|\leq 1$, $|y|\leq 1$). He calls subsets of $\mathbb Z^2$ (or, more generally, of
$\mathbb Z^n$) which have this property {\em uniformly distributed}. In a subsequent paper we will discuss
uniform distribution of sets of the type $S(T)$.

We end this section with an application of our theorems.

\subsection{The probability that $k$ random elements generate the group $\mathbb Z^n$}\label{kap}
In his work on generic properties of one-relator groups Ilya Kapovich was led to the following
question: what is the probability that several randomly chosen elements generate the group
$\mathbb{Z}^n$. Even though there is a fairly simple heuristic argument which
leads to an answer, neither Kapovich nor we have been able to find a reference containing a
proof.
The techniques developed in this paper allow in particular to give a rigorous answer to Kapovich's
question. The key observation is contained in the following lemma.

\begin{lemma}\label{qcomb}
Let $V$ ba an $n$-dimensional vector space over the finite field $\Fg_q$. The number
$\alpha_{m,n}=\alpha_{m,n}(q)$ of
$m$-tuples of elements in $V$ that span $V$ is equal to $\prod_{i=0}^{n-1}(q^m-q^i)$.
\end{lemma}

\begin{proof}
For $m<n$ the formula is obviously true as it yields $0$ and there are no $m$-tuples
which span $V$.
The number $\alpha_{n,n}$
equals the number of bases of $V$, which is well known to be equal to
$|\text{\rm GL}_n(\mathbb F_{q})|=\prod_{i=0}^{n-1}(q^n-q^i)$. This establishes the result
for $m=n$.
Note now that $v_1,\ldots,v_m$ span $V$ if and only if the images of
$v_2,\ldots,v_m$ in $V/<v_1>$ span $V/<v_1>$.
Given $v\in V$, we count the number of $m$-tuples which span $V$ and start with
$v_1=v$.
If $v=0$ this number is clearly $\alpha_{m-1,n}$. If $v\neq 0$, then there are
$\alpha_{m-1,n-1}$
$(m-1)$-tuples which span $V/<v>$ and each such tuple lifts to $q^{m-1}$
$(m-1)$-tuples from $V$.
Thus we get $q^{m-1}\alpha_{m-1,n-1}$ $m$-tuples which span $V$ and start at $v$.
Since there are $q^n-1$
non-zero elements in $V$, we get the following recursive formula:
\[\alpha_{m,n}=\alpha_{m-1,n}+q^{m-1}(q^n-1)\alpha_{m-1,n-1}.\]
The recursive formula and a straightforward induction on $m+n$ finish the proof.
\end{proof}

\begin{theorem}\label{ilia}
Let $R$ be an order in a number field. Define
\[\displaystyle \zeta_R(s)=\prod_{ \f p \in
\mspec(R)}\left(1-\left|R/\f p\right|^{-s}  \right)^{-1}.\]
For any $k\geq n$ the density of the set of $k$ tuples
which generate the $R$-module $R^n$ is equal to
$\displaystyle \prod_{m=k-n+1}^{k}
\zeta_R(m)^{-1}$.
\end{theorem}

\begin{proof}
Consider $R^n$ as an $R$-algebra with trivial multiplication and let $A$ be obtained from $R^n$
by the construction of adjunction of unity (in the category of $R$-algebras). By Lemma~\ref{add1} we see that
the density of the set of $k$-tuples which generate the $R$-module $R^n$ is the same as the density $\den_k(A)$
of the set of $k$-tuples which generate the $R$-algebra $A$. By Lemmas~\ref{add1}
and \ref{add1mod}, we have $g_k(\f p, A)=\N(\f p)^k\alpha_{k,n}(\N(\f p))$. By Theorem~\ref{gendensity}
and Lemma~\ref{qcomb} we obtain the formula
\[ \den_k(A)=
\prod_{\f p\in \mspec R}\frac{
\prod_{i=0}^{n-1}(\N(\f p)^k-\N(\f p)^i)}{\N(\f p)^{nk}}= \prod_{m=k-n+1}^{k}
\zeta_R(m)^{-1}.\]
\end{proof}
The answer to Ilya Kapovich's question is therefore given by the following corollary.

\begin{corollary}
The probability that $k$ randomly chosen elements generate the group
$\mathbb{Z}^n$ is equal to $
\prod_{m=k-n+1}^{k} \zeta(m)^{-1}$, where $\zeta$ is the
Riemann zeta-function.
\end{corollary}
Note that this corollary can also be derived directly from Theorem~\ref{main}.

\section{Proof of Theorem \ref{main}}
Let us start by recalling some of the notation set in the previous section.
$R$ is an order in a number field $K$. The degree of $K$ over $\mathbb Q$
is $d$ and $O_K$ is the ring of integers of $K$ (i.e. the integral closure of $R$ in $K$).
We fix an integral basis $w_1,\ldots,w_d$ of
$R$ over $\mathbb Z$. Any element $r$ of $R$ can be uniquely
written as $r=\sum_{i=1}^d r_i w_i$ with $r_i\in \mathbb Z$. For a
positive integer $N$ we denote by $R(N)$ the set of all $r\in R$
such that $|r_i|\leq N$ for all $i$. Clearly $|R(N)|=(2N+1)^d$.
The norm map from $K$ to $\mathbb Q$ is denoted by $\N_{K/\mathbb Q}$.
For an ideal $I$ of $R$ we set $\N_{K/\mathbb Q}(I)$ for the non-negative
integer which is the greatest common divisor of the norms of all elements in $I$.
We write $\N(I)$ for the cardinality of $R/I$.
If $\f p$ is a maximal ideal of $R$ then we write $\mathbb F_{\f p}$ for the field $R/\f p$.

The following lemma is well known but for the readers convenience
we include a short proof.
\begin{nle}\label{zeros}
Let $F$ be a finite field with $q$ elements and let
$f(x_1,\ldots,x_n)$ be a non-zero polynomial in $F[x_1,\ldots,x_n]$. Then the number of solutions of the
equation
$f(x_1,\ldots,x_n)=0$ in $F^n$ does not exceed $(\deg f) q^{n-1}$.
\end{nle}

\begin{proof}
We proceed by induction on $n$. For $n=1$ this is just the
statement that a polynomial $f$ in one variable over a field has
at most $\deg f$ roots. Suppose now that the result holds for polynomials in less than $n$ variables
and let $f(x_1,\ldots,x_n)=
\sum_{i=0}^{d}f_i(x_1,\ldots,x_{n-1})x_n^i$ be a polynomial in $n$ variables with $f_d\neq 0$. By the inductive assumption,
the number of solutions to $f_d=0$ in $F^n$ does not exceed $(\deg f_d)\cdot
q^{n-2}\cdot q=(\deg f_d) q^{n-1}$. For each $(a_1,\ldots,a_{n-1})\in
F^{n-1}$ such that $f_d(a_1,\ldots,a_{n-1})\neq 0$ there are at most $d$ solutions to
$f(a_1,\ldots,a_{n-1},x_n)=0$. Thus we have at most $(\deg f_d)
q^{n-1}+dq^{n-1}\leq (\deg f) q^{n-1}$ solutions to $f=0$.
\end{proof}

\begin{pro}\label{hyper}
Let $f(x_1,\ldots,x_n)\in \mathbb Z[x_1,\ldots,x_n]$ be a non-zero polynomial. Define
\[Z(f,N)=\{(x_1,\ldots,x_n)\in\mathbb Z^n: |x_i|\leq N \ \text{and}\ f(x_1,\ldots,x_n)=0\}. \]
Then $|Z(f,N)|\leq (\deg f)(2N+1)^{n-1}$.
\end{pro}
\begin{proof}
We proceed by induction on $n$. For $n=1$ the result
is straightforward. Now assume the results for polynomials in
less than $n$ variables and consider a polynomial $f(x_1,\ldots,x_n)=
\sum_{i=0}^{d}f_i(x_1,\ldots,x_{n-1})x_n^i$ in $n$ variables with $f_d\neq 0$. By
the inductive assumption, the
number of elements in $Z(f,N)$ for which
$f_d=0$ does not exceed $(\deg f_d)\cdot
(2N+1)^{n-2}\cdot(2N+1)$. For each $(a_1,\ldots,a_{n-1})$ such that
$f_d(a_1,\ldots,a_{n-1})\neq 0$ there are at most $d$ solutions to
$f(a_1,\ldots,a_{n-1},x_n)=0$. Thus we have at most $(\deg f_d)
(2N+1)^{n-1}+d(2N+1)^{n-1}\leq (\deg f) (2N+1)^{n-1}$ elements in $Z(f,N)$.
\end{proof}

\begin{cl}\label{bounded}
Let $f\in R[x_1,\ldots,x_n]$ be a polynomial of positive degree $\deg f>0$. For a non-zero ideal $J$ of $R$
define
\[ I(f,J,N)=\{(x_1,\ldots,x_n)\in R(N)^n: J\subseteq f(x_1,\ldots,x_n)R\} .\]
Then $|I(f,J,N)|\leq \delta(J)d(\deg f)(2N+1)^{dn-1}$, where $\delta(J)$ is the number of integral divisors
of the norm $\N_{K/\mathbb Q}(J)$.
\end{cl}
\begin{proof}
Write $x_i=\sum_{j=1}^{d} y_{i,j}w_j$. Note that if $J\subseteq f(x_1,\ldots,x_n)R$ then
$\N_{K/\mathbb Q}(f(x_1,\ldots,x_n))$ divides $\N_{K/\mathbb Q}(J)$. There is
a polynomial $g(y_{i,j})\in \mathbb Z[y_{1,1}, y_{1,2},\ldots, y_{n,d}]$ of degree $(\deg f)d$ in $dn$ variables such that
\[ \N_{K/\mathbb Q}(f(x_1,\ldots,x_n))=g(y_{i,j}).\]
The result follows now from Proposition~\ref{hyper} applied to each of the polynomials $g-k$, where
$k$ varies over all divisor of $\N_{K/\mathbb Q}(J)$.
\end{proof}

\begin{te}\label{split}
Let $f\in R[x_1,\ldots,x_n]$ be a polynomial of positive degree. For each maximal ideal $\f p$ of $R$ let
$f_{\f p}$ be the number of solutions to $f=0$ over the finite field
$\mathbb F_{\f p}$. Then the series $\displaystyle \sum_{\f p\in \mspec R} f_{\f p}/\N(\f p)^n$ diverges.
\end{te}
\begin{proof}
Replacing $R$ by $O_K$ changes only a finite number of terms in the sum
$\sum_{\f p} f_{\f p}/\N(\f p)^n$. It suffices then to prove the theorem under the additional
assumption that $R=O_K$.

Let $L$ be a number field containing $K$, with ring of integers $S$, and
such that $f$ has an absolutely irreducible divisor $g\in
S[x_1,\ldots,x_n]$ of positive degree (so $f/g\in L[x_1,\ldots,x_n]$). It is known that the reduction of
$g$ modulo all but a finite number of prime ideals of $S$ is absolutely irreducible (see
\cite[V.2]{schmidt}). For a maximal ideal $P$ of $S$ let $g_P$ be
the number of solutions to $g=0$ in $(S/P)^{n}$. By \cite[V, Theorem 5A]{schmidt},
$g_P\geq \frac{1}{2}|S/P|^{n-1}=\frac{1}{2}\N(P)^{n-1}$ provided  the reduction of $g$ modulo $P$ is absolutely
irreducible and $\N(P)$ is sufficiently large, which holds for all but a finite number
of maximal ideals of $S$.

Let $\Phi$ be the set of maximal ideals of $S$ which have inertia degree 1
over $R$ and let $\Psi$ be the set of all prime ideals of $R$ which lie under the ideals of $\Phi$.
Let $P\in\Phi$ be a prime ideal of $S$ over
$\f p\in \Psi$. Then $S/P=\mathbb F_{\f p}$. It follows that $f_{\f p}\geq g_{P}$ except possibly for a finite
number of $P$ for which $f/g$ is not $P$-integral.
Since each maximal ideal of $R$ lies under at most $[L:K]$ prime ideals of $S$ we get that
\[\sum_{\f p\in \mspec R} \frac{f_{\f p}}{\N(\f p)^n}\geq \sum_{\f p\in \Psi} \frac{f_{\f p}}{\N(\f p)^n}\geq
\frac{1}{[L:K]} \sum_{P\in \Phi, \N(P)>>0} \frac{g_P}{\N(P)^n}\geq \]
\[\geq \frac{1}{2[L:K]}\sum_{P\in \Phi, \N(P)>>0}\frac{1}{\N(P)}.\]
It is well known that the set $\Phi$ has Dirichlet density equal to $1$ (\cite[7.2, Corollary 3]{nar}),
in particular $\sum_{P\in \Phi}1/\N(P)$ diverges.
\end{proof}

\begin{cl} \label{zero} Under the assumptions of Theorem~\ref{main},
if the polynomials in $T$ have a common divisor of positive degree in $K[x_1,\ldots,x_n]$ then both sides of
{\rm (\ref{denformula})} are $0$. In particular, Theorem~\ref{main} is true in this case.
\end{cl}
\begin{proof}
Let $f\in R[x_1,\ldots,x_n]$ be a polynomial of positive degree which divides
all $f_i$ in the ring $K[x_1,\ldots,x_n]$. There is a non-zero $a$ in $R$ such that $af_i/f\in R[x_1,\ldots,x_n]$ for all $i$.
It follows that $S_N(T)\subseteq I(f,aR, N)$. By Corollary~\ref{bounded},
\[\frac{|S_N|}{(2N+1)^{dn}}\leq \frac{|I(f,aR,N)|}{(2N+1)^{dn}}\leq \frac{\delta(aR)d(\deg f)}{2N+1},\]
so the left hand side of (\ref{denformula}) is 0.

For any maximal ideal  $\f p$ of $R$ which does not divide $a$ we have $t_{\f p}\geq f_{\f p}$.
It follows from Theorem \ref{split} that $\sum_{\f p\in\mspec R} t_{\f p}/\N(\f p)^n$ diverges.  This
is equivalent to the right hand side of (\ref{denformula}) being 0.
\end{proof}

\begin{nle}\label{residues}
Let $\f p$ be a maximal ideal of $R$ with $\N(\f p)=p^s$, where $p$ is the characteristic
of $\mathbb F_{\f p}$. Then any element of $\mathbb F_{\f p}$ lifts to at most $(2N+1)^{d-s}(1+2N/p)^s$ elements
in $R(N)$.
\end{nle}
\begin{proof} We may assume (after renumbering, if necessary) that $w_1,\ldots,w_s$ is a basis of
$\mathbb F_{\f p}$ over $\mathbb F_{p}$. Consider a residue class $a\in \mathbb F_{\f p}$.
To get an element $\sum_{i=1}^d y_iw_i\in R(N)$
in the given residue class $a$ we may choose arbitrarily  integers
$y_{s+1},\ldots,y_d\in [-N,N]$ and then the residue classes of
$y_1,\ldots,y_s$ modulo $p$ are uniquely determined. Thus each $y_i$, $i\leq s$, can be chosen in at most
$1+2N/p$ ways.
\end{proof}

%The main step in our proof of Theorem 1 is the following

\begin{nle}\label{big}
Let $f\in R[x_1,\ldots,x_{n-1}]$, $g=g_0x_n^k+\ldots+g_k\in
R[x_1,\ldots,x_n]$, where $g_0,\ldots,g_k\in
R[x_1,\ldots,x_{n-1}]$ and $g_0\neq 0$. Consider the set
\[D(N)=\{(x_1,\ldots,x_n)\in R(N)^n: f(x_1,\ldots,x_{n-1})\neq 0 \ \text{ and there exists a} \]
\[\text{maximal ideal $\f p$ with $ \N(\f p)>N$ and such that} \
f(x_1,\ldots,x_{n-1})\in \f p,\]
\[ g(x_1,\ldots,x_n)\in\f p,\ \text{and}\   g_0(x_1,\ldots,x_{n-1})\not\in \f p\}.\]
Then there is a constant $c$ such that $|D(N)|\leq c(2N+1)^{dn-1}$ for all $N$.
%Then $\lim_{N\to\infty}|B(N)|/(2N+1)^n=0$.
\end{nle}

\begin{proof} There are positive integers $w$, $c_1$ such that
\[ |\N_{K/\mathbb Q}f(x_1,\ldots,x_{n-1})|\leq c_1N^w\]
for any $N\geq 1$ and any $x_i\in R(N)$, $i=1,\ldots,n-1$. If $N>c_1$ and $f(x_1,\ldots,x_{n-1})\neq 0$, then
$f(x_1,\ldots,x_{n-1})$ belongs to at most $w$ maximal ideals $\f p$ such that $\N(\f  p)>N$.
In fact, if there were more than $w$ maximal ideals in $R$ with norm exceeding $N$ which contain
$f(x_1,\ldots,x_{n-1})$ then $f(x_1,\ldots,x_{n-1})$ would belong to at least $w+1$ maximal ideals
of $O_K$ of norm exceeding $N$ and this would imply that $|\N_{K/\mathbb Q}f(x_1,\ldots,x_{n-1})|>N^{w+1}$,
which is not possible. Let
\[G(N,\f p)=\{(x_1,\ldots,x_{n-1})\in R(N)^{n-1}:  f(x_1,\ldots,x_{n-1})\in \f p-\{0\}\}.\]
Thus, if $N>c_1$ and $(x_1,\ldots,x_{n-1})\in R(N)^{n-1}$, then there are at most
$w$ maximal ideals $\f p$ such that $\N(\f p)>N$ and $(x_1,\ldots,x_{n-1})\in G(N,\f p)$. Let $N>c_1$.
Fix a point $(x_1,\ldots,x_{n-1})\in R(N)^{n-1}$ and let $\f p$ be a maximal
ideal such that $\N(\f p)>N$ and $(x_1,\ldots,x_{n-1})\in G(N,\f p)$. We want to find an upper bound
for the number of $x_n\in R(N)$ such that $g(x_1,\ldots,x_n)\in\f p$ and $g_0(x_1,\ldots,x_{n-1})\not\in \f p$.
All such $x_n$ split into at most $k$ residue classes modulo $\f p$ (which correspond to the roots
of $g(x_1,\ldots,x_{n-1},x)=0$ in $\mathbb F_{\f p}$).
Let $\N(\f p)=p^s$, where $p$ is the characteristic of $\mathbb F_{\f p}$.
By Lemma~\ref{residues}, the number of $x_n\in R(N)$ which belong to a given
residue class modulo $\f p$ is at most
\[(2N+1)^{d-s}\max(2,2(2N+1)/p)^s\leq \max(2^s(2N+1)^{d-s}, 2^s(2N+1)^s/p^s)\leq\]
\[\leq  3\cdot 2^s\cdot(2N+1)^{d-1} \]
(we use the inequalities $1+2N/p\leq \max(2,2(2N+1)/p)$ and  $p^s>N\geq (2N+1)/3$).
It follows that there are at most $w\cdot k\cdot 3\cdot 2^s (2N+1)^{d-1}$ values of $x_n\in R(N)$
such that $(x_1,\ldots,x_n)\in D(N)$. Hence, if $N>c_1$, then
\[|D(N)|\leq (2N+1)^{d(n-1)}\cdot w\cdot k \cdot 3\cdot 2^s\cdot (2N+1)^{d-1}\leq c(2N+1)^{dn-1},\]
where $c=3\cdot 2^s\cdot w\cdot k$. We can increase $c$ if necessary so that the inequality
$|D(N)|\leq c(2N+1)^{dn-1}$ holds for all $N$.
\end{proof}

\begin{nle}\label{medium}
Let $f$ be a non-zero polynomial in $R[x_1,\ldots,x_{n-1}]$ and let
$g=g_0x_n^k+\ldots+g_k\in R[x_1,\ldots,x_n]$, where
$g_0,\ldots,g_k\in R[x_1,\ldots,x_{n-1}]$, $g_0\neq 0$. For a
maximal ideal $\f p$ of $R$ consider the set
\[D_{\f p}(N)=\{(x_1,\ldots,x_n)\in R(N)^n: f(x_1,\ldots,x_{n-1})\in \f p, \]
\[g(x_1,\ldots,x_n)\in \f p, \ \text{and}\ g_0(x_1,\ldots,x_{n-1})\not\in \f p\}\]
Then, if $\N(\f p)\leq N$ and the reduction of $f$ modulo $\f p$ is not zero, we have
\[|D_{\f p}(N)|\leq 2^{nd}(\deg f)k(2N+1)^{nd}/\N(\f p)^2.\]
\end{nle}
\begin{proof} The image $Z_{\f p}$ of $D_{\f p}(N)$ in $\mathbb F_{\f p}^n$
consists of (some) solutions to $f=0=g$ in $\mathbb F_{\f p}^n$ (we use the same notation for a polynomial
and its reduction modulo $\f p$) . Now
$f=0$ has at most $(\deg f)\N(\f p)^{n-2}$ solutions in $\mathbb
F_p^{n-1}$ (Lemma \ref{zeros}) and each such solution extends to at most
$k$ solutions of $g=0$, $g_0\neq 0$ in $\mathbb F_{\f p}^n$. Thus
$|Z_{\f p}|\leq (\deg f)k\N(\f p)^{n-2}$. Each
element of $Z_{\f p}$ lifts to no more that $[(2N+1)^{d-s}(1+2N/p)^s]^n$ elements of
$D_{\f p}(N)$ by Lemma \ref{residues}, where $\N(\f p)=p^s$. Thus
\[|D_{\f p}(N)|\leq (\deg f)k\N(\f p)^{n-2}[(2N+1)^{d-s}(1+2N/p)^s]^n\leq \]
\[\leq  (\deg f)k\N(\f p)^{n-2}(2N+1)^{n(d-s)}[2^s(2N+1)^s/p^s]^n \leq \]
\[\leq  2^{nd}(\deg f)k\N(\f p)^{n-2}(2N+1)^{nd}/\N(\f p)^n =
2^{nd}(\deg f)k(2N+1)^{nd}/\N(\f p)^2.\]
\end{proof}

\begin{pro}\label{mainpro}
Let $f,g\in R[x_1,\ldots,x_n]$ be polynomials which are
relatively prime as polynomials in $K[x_1,\ldots,x_n]$.
Define
\[ W_M=W_M(f,g)=\{{\mathbf r}=(r_1,\ldots,r_n)\in R^n: \text{ there is a maximal  ideal $\f p$}\]
\[ \text{of $R$, with $\N(\f p) >M$ and such that } f({\mathbf r})\in \f p \ \text{and}\  g({\mathbf r})\in \f p\}.\]
There is a constant $c>0$ such that
\[|W_M\cap R(N)^n|\leq c\frac{(2N+1)^{nd}}{M}\]
for any integers $N>M\geq 1$.
\end{pro}
\begin{proof}
We use induction on the number $n$ of variables. Note that if $f,g$ are polynomials in $n$ variables
for which the result holds, then it also holds for $f,g$ considered as polynomials in $n+1$ variables.
When $n=0$ the result is clear. Suppose the result is true for less
than $n\geq 1$ variables. Consider two relatively prime (in $K[x_1,\ldots,x_n]$) polynomials
$f,g\in R[x_1,\ldots,x_n]$

The first step is to establish the proposition under the additional assumption that
$f$ is irreducible in $K[x_1,\ldots,x_n]$ and does not depend on $x_n$ (i.e. $f\in R[x_1,\ldots,x_{n-1}]$).
Let $g=g_0x_n^k+\ldots+g_k$, where
$g_0,\ldots,g_k\in R[x_1,\ldots,x_{n-1}]$, $g_0\neq 0$. We fix $f$ and proceed by induction on the degree $k$
of $g$ in $x_n$. If $k=0$ then
$g\in R[x_1,\ldots,x_{n-1}]$ and the result follows by our inductive assumption that the proposition
holds for polynomials in $n-1$ variables. Suppose that $k>0$ and the result holds for all
polynomials $g$ whose degree in $x_n$ is less than $k$ (and which are relatively prime to $f$).
We may write $ag=\prod h_i$ for some non-zero $a$ in $R$ and polynomials $h_i \in R[x_1,\ldots,x_{n}]$  which are
irreducible in $K[x_1,\ldots,x_{n}]$. Note that $W_M(f,g)\subseteq \bigcup W_M(f,h_i)$. Thus, if we show the
proposition for each pair $f, h_i$, then it will also hold for the pair $f,g$.  In other words,
we may assume that $g$ is irreducible in $K[x_1,\ldots,x_{n}]$.
If $f|g_0$ in $K[x_1,\ldots,x_{n-1}]$, then there is a non-zero $u\in R$ such that $f|ug_0$ in
$R[x_1,\ldots,x_{n-1}]$. It follows that
$W_M(f,g)\subseteq W_M(f,u(g-g_0x_n^k))$ for all $M$. Since $u(g-g_0x_n^k)$ has
degree in $x_n$ smaller than $k$, the result holds for $f,u(g-g_0x_n^k)$ by our inductive assumption and
therefore it also holds for the pair $f, g$.
Thus we may assume that $f$ does not divide $g_0$ in $K[x_1,\ldots,x_{n-1}]$. Since $f$ is irreducible,
$f$ and $g_0$ are relatively prime in $K[x_1,\ldots,x_{n-1}]$. For $N> M$ we have
\[ W_M(f,g)\cap R(N)^n\subseteq \]
\[(W_M(f,g_0) \cap R(N)^n)\cup Z(f, N)\cup D(N) \cup \bigcup_{\f p:M < \N(\f p)\leq N}
D_{\f p}(N), \]
where
\[Z(f, N)=\{{\mathbf r}=(r_1,\ldots,r_n)\in R(N)^n: f({\mathbf r})=0 \}, \]
\[ D(N)=\{{\mathbf r}\in R(N)^n: \ \text{there is a maximal ideal $\f p$ such that $ \N(\f p)>N$, }\]
\[\
f({\mathbf r})\in \f p-\{0\}, g({\mathbf r})\in\f p, \ \text{and}\  g_0({\mathbf r})\not\in \f p\},\]
\[D_{\f p}(N)=\{{\mathbf r}\in R(N)^n: f({\mathbf r})\in \f p, g({\mathbf r})\in\f p,  g_0({\mathbf r})
\not\in \f p\}.\]
By our inductive assumption that the proposition holds for polynomials in $n-1$ variables, there is $c_1>0$
such that $|W_M(f,g_0)\cap R(N)^n|\leq c_1(2N+1)^{dn}/M$ for any integers $N>M\geq 1$.
Note that if $f({\mathbf r})=0$ then $(f-1)({\mathbf r})R=R$. It follows by Corollary \ref{bounded}
applied to the polynomial $f-1$ and the ideal $J=R$
that
\[|Z(f, N)|\leq \delta(R)d(\deg f)(2N+1)^{dn-1}\leq c_2 \frac{(2N+1)^{dn}}{M}\]
for some $c_2>0$ and all $N>M\geq 1$.
Lemma \ref{big} assures the existence of $c_3>0$ such that $|D(N)|\leq c_3 (2N+1)^{dn-1}\leq c_3 (2N+1)^{dn}/M$.
Finally, by
Lemma \ref{medium}, there are constants $c_4>0$, $c_5>0$ such that
\[|\bigcup_{\f p:M < \N(\f p)\leq N}D_{\f p}(N)|\leq \sum_{\f p:M < \N(\f p)\leq N}|D_{\f p}(N)|\leq\]
\[\leq \sum_{\f p:M < \N(\f p)\leq N}2^{nd}(\deg f)d\frac{(2N+1)^{nd}}{\N(\f p)^2} \leq
c_4(2N+1)^{nd}\sum_{\f p:M < \N(\f p)}\N(\f p)^{-2}\leq
\]
\[\leq c_4(2N+1)^{nd}d\sum_{m>M}\frac{1}{m^2}\leq
c_5\frac{(2N+1)^{nd}}{M}.\]
It follows that $|W_M(f,g)\cap R(N)^n|\leq c(2N+1)^{nd}/M$, where $c=c_1+c_2+c_3+c_5$.
This completes our first step, i.e. establishes the proposition under the additional assumption that
$f$ is irreducible in $K[x_1,\ldots,x_n]$ and does not depend on $x_n$.

Our second step is to prove the proposition when both $f$ and $g$ are
irreducible in $K[x_1,\ldots,x_n]$. Consider $f,g$ as polynomials in $x_n$ with
coefficients in $R[x_1,\ldots,x_{n-1}]$. If one of these polynomials does not depend on $x_n$, the proposition
holds by our first step. Suppose that the degrees
with respect to $x_n$ of both $f$ and $g$ are positive. Let
$r=\text{Res}(f,g)$ be the resultant of $f$ and $g$, so $r$ is a non-zero polynomial in
$R[x_1,\ldots,x_{n-1}]$. Recall that $r=af+bg$ for some polynomials
$a,b\in R[x_1,\ldots,x_n]$ (see \cite[Section 3.1]{cox} for a nice account of properties
of resultants). It follows that
$W_M(f,g)\subseteq W_M(f,r)\cap W_M(g,r)$. Since $f$ and $g$ are
irreducible, $g$ and $r$ have no common factor in
$K[x_1,\ldots,x_n]$ (otherwise $g$ would not depend on $x_n$).
We may write $ar=\prod r_i$, where $r_i \in R[x_1,\ldots,x_{n-1}]$ are irreducible
in $K[x_1,\ldots,x_{n-1}]$ and $a\in R$ is non-zero. Clearly $W_M(f,g)\subseteq W_M(r,g)\subseteq \bigcup W_M(r_i,g)$.
Since the proposition holds for each pair $r_i, g$ by the first step, it also holds for the pair $f,g$.

Finally, without any additional assumptions,
we may write $af=\prod f_i$, $bg=\prod
g_i$, where $f_i, g_j\in R[x_1,\ldots,x_n]$ are irreducible
in $K[x_1,\ldots,x_n]$ and $a, b\in R-\{0\}$. Clearly $W_M(f,g)\subseteq \bigcup W_M(f_i,g_j)$.
Since the result holds for each pair $f_i,g_j$, it also holds for $f,g$.

\end{proof}

\begin{cl}\label{Main}
Let $T=\{f_1,\ldots,f_s\}$  be a finite set of polynomials in $R[x_1,\ldots,x_n]$ which do not have any
common non-constant divisor in $K[x_1,\ldots,x_n]$. Define
\[ W_M=W_M(T)=\{{\mathbf r}=(r_1,\ldots,r_n)\in R^n: \text{ there is a maximal ideal $\f p$} \]
\[ \text{of $R$ with $\N(\f p) >M$ and such that } f({\mathbf r})\in \f p \ \text{for every $f\in T$}\}.\]
There is a constant $c>0$ such that $|W_M\cap R(N)^n|\leq c(2N+1)^{nd}/M$ for any integers $N>M\geq 1$.
\end{cl}
\begin{proof} We may write $d_if_i=\prod f_{i,j}$, where $f_{i,j}\in R[x_1,\ldots,x_n]$ are irreducible
in $K[x_1,\ldots,x_n]$ and $d_i\in R$ are non-zero. Then
\[W_M\subseteq \bigcup W_M(f,g),\]
where the union is over all pairs $f,g$
such that both $f$ and $g$ are among the polynomials $f_{i,j}$ and they are relatively prime.
Thus the result follows by Proposition \ref{mainpro}.
\end{proof}

Corollary~\ref{Main} is the main ingredient in our proof of Theorem~\ref{main}. In fact, the proof now
reduces to a fairly straightforward application of the inclusion-exclusion formula and the
Chinese Remainder Theorem. For the benefit of the reader we provide a detailed argument.

\begin{nle}\label{ideals}
Let $I$ be a non-zero ideal of $R$. If $m$ is a positive integer  such that $mR\subseteq I$
then
\[ \frac{(2N-m)^d}{\N(I)}\leq |(a+I)\cap R(N)|\leq \frac{(2N+m)^d}{\N(I)} \]
for any $a\in R$ and any $N$ such that $2N\geq m$.
\end{nle}
\begin{proof} The ideal $I$ is a union
of $m^d/\N(I)$ cosets of $mR$. Thus any coset of $I$ is also a union of $m^d/\N(I)$ cosets
of $mR$. Any coset $H$ of $mR$ is of the form $\sum_{i=1}^{d} a_iw_i+mR$, where $0\leq a_i<m$.
Elements of $H\cap R(N)$ are exactly the elements of the form $\sum_{i=1}^{d} (a_i+mb_i)w_i$
with $|a_i+mb_i|\leq N$. Thus $(-N-a_i)/m\leq b_i\leq (N-a_i)/m$. Recall now that an
interval of length $l$ has at least $l-1$ and at most $l+1$ integers in it.
It follows that $(2N/m-1)^d\leq |H\cap R(N)|\leq (2N/m+1)^d$. Since $a+I$
is a disjoint union of $m^d/\N(I)$ cosets of $mR$, the result follows.
\end{proof}

\begin{nle}\label{ideals1}
Let $I$ be a non-zero ideal of $R$. If $V$ is a subset of $(R/I)^n$ and $V(N)$ is the set of elements
of $R(N)^n$ whose image in $(R/I)^n$ belongs to $V$ then
\[ \lim_{N\to \infty} \frac{|V(N)|}{(2N+1)^{nd}}=\frac{|V|}{\N(I)^n}.\]
\end{nle}
\begin{proof}
Since both sides of the equality are additive for disjoint unions,  it suffices to prove the lemma for
sets $V$ which contain only one element. In this case, there are
cosets $a_1+I,\ldots, a_n+I$ of $I$ such that $V(N)=((a_1+I)\cap R(N))\times\ldots\times ((a_n+I)\cap R(N))$.
There is a positive integer $m$ such that $mR\subseteq I$.
By Lemma~\ref{ideals}, we have
\[ \frac{(2N-m)^{dn}}{\N(I)^n}\leq |V(N)|\leq \frac{(2N+m)^{dn}}{\N(I)^n}\]
provided $2N\geq m$.
Dividing by $(2N+1)^{dn}$ and passing to the limit when $N\to \infty$, we get the result.
\end{proof}

\begin{proof}[{\rm {\bf Proof of Theorem \ref{main}.}}] If the polynomials in $T$ have a common divisor in $K[x_1,\ldots,x_n]$
the theorem holds by Corollary \ref{zero}. Thus we may assume that elements of $T$ do not have any
common non-constant divisor in $K[x_1,\ldots,x_n]$. For a prime ideal $\f p$ of $R$ define
\[ D_{\f p}=\{{\mathbf r}=(r_1,\ldots,r_n)\in R^n:   f({\mathbf r})\in \f p \ \text{for every $f\in T$}\}.\]
%and set $ D_{\f p}(N)= D_{\f p}\cap R(N)^n$.
Let $\Phi$ be a finite set of maximal ideals of $R$. For any subset $\Psi$ of $\Phi$ we denote by $I(\Psi)$
the intersection of all the ideals in $\Psi$. Note that
\[D_{\Psi}:= \bigcap _{\f p\in \Psi}D_{\f p}=\{{\mathbf r}=(r_1,\ldots,r_n)\in R^n:
f({\mathbf r})\in I(\Psi) \ \text{for every $f\in T$}\}.\]
Let $V_{\Psi}$ be the image of $D_{\Psi}$ in $(R/I(\Psi))^n$. Thus $V_{\Psi}$ is simply the set of all common
zeros in  $(R/I(\Psi ))^n$ of the polynomials in $T$.
By the Chinese Remainder Theorem, we have $R/I(\Psi)\cong \prod_{\f p\in \Psi}R/\f p$ and under this
identification we have $V_{\Psi}=\prod_{\f p\in \Psi}V_{\f p}$.
It follows that $|V_{\Psi}|=\prod_{\f p \in \Psi}t_{\f p}$.
Applying Lemma \ref{ideals1} to the set $V_{\Psi}$ and observing that $V_{\Psi}(N)=D_{\Psi}\cap R(N)^n$ we get
\[ \lim_{N\to \infty}\frac{|D_{\Psi}\cap R(N)^n|}{(2N+1)^{nd}}=\frac{\prod_{\f p \in \Psi}t_{\f p}}{\N(I(\Psi))^n}=
\prod_{\f p \in \Psi}\frac{t_{\f p}}{\N(\f p)^n}.\]
Let $W_{\Phi}$ be the complement of the union $ \bigcup _{\f p\in \Phi}D_{\f p}$ in $R^n$. The inclusion-exclusion principle yields the following formula:
\[ |W_{\Phi}\cap R(N)^n|=\sum_{\Psi\subseteq \Phi}(-1)^{|\Psi|}|D_{\Psi}\cap R(N)^n|\]
 (where $D_{\emptyset}=R^n$), from which we immediately conclude that
\[  \lim_{N\to \infty}\frac{|W_{\Phi}\cap R(N)^n|}{(2N+1)^{nd}}=\sum_{\Psi\subseteq \Phi}(-1)^{|\Psi|}
\prod_{\f p \in \Psi}\frac{t_{\f p}}{\N(\f p)^n}=\prod_{\f p \in \Phi}\left(1-\frac{t_{\f p}}{\N(\f p)^n}\right). \]
Suppose now that $\Phi$ is the set of all prime ideals of norm $\leq M$. Note that
$S(T)\subseteq W_{\Phi}\subseteq S(T)\cup W_M(T)$, where $W_M(T)$ is defined in  Corollary~\ref{Main}
and $S(T)$ in the statement of Theorem~\ref{main}.
Thus
\[  |W_{\Phi}\cap R(N)^n|- |W_M(T)\cap R(N)^n|\leq|S(T)\cap R(N)^{n}|\leq |W_{\Phi}\cap R(N)^n|.\]
Note that Corollary~\ref{Main} implies that
\[ 0\leq \liminf_{N\to \infty} \frac{|W_M(T)\cap R(N)^n|}{(2N+1)^{dn}}\leq  \limsup_{N\to \infty}
\frac{|W_M(T)\cap R(N)^n|}{(2N+1)^{dn}}\leq \frac{c}{M}. \]
This yields
\[ \prod_{\f p: \N(\f p)\leq M}\left(1-\frac{t_{\f p}}{\N(\f p)^n}\right)-\frac{c}{M}\leq  \liminf_{N\to \infty}\frac{ |S(T)\cap R(N)^{n}|}{(2N+1)^{dn}}\leq \]
\[\limsup_{N\to \infty} \frac{|S(T)\cap R(N)^{n}|}{(2N+1)^{dn}}
\leq \prod_{\f p: \N(\f p)\leq M}\left(1-\frac{t_{\f p}}{\N(\f p)^n}\right).\]
Letting $M$ go to infinity we see that
\[\lim_{N\to \infty} \frac{|S(T)\cap R(N)^{n}|}{(2N+1)^{dn}} = \prod_{\f p\in\mspec R}\left(1-\frac{t_{\f p}}{\N(\f p)^n}
\right).\]
\end{proof}

\section{The smallest number of generators}
Let us return to our discussion of generators of algebras. We first show an application of
Theorem \ref{gendensity}. Let $A$ be an algebra over a commutative ring $R$, which is finitely
generated as an $R$-module.

\begin{definition}{\rm
$r=r(A)=r(A,R)$ is the smallest number of elements needed to generate
$A$ as an $R$-algebra.

\noindent
For a prime ideal $\f p$ of $R$ define
\[r_{\f p}=r_{\f p}(A)= r(A\otimes_R \kappa(\f p),\kappa(\f p)),\]
where
$\kappa(\f p)=R_{\f p}/\f pR_{\f p}$ is the field of fractions of $R/\f p$.}
\end{definition}

Note that $r_{\f p}$
is the smallest number of generators of $A_{\f p}$ as an $R_{\f p}$-algebra by Lemma \ref{localgen}. Clearly
$r_{\f p}\leq r$ for every $\f p\in \spec R$ and $r_{\f p}\leq r_{\f q}$ whenever $\f p\subseteq \f q$
by Corollary \ref{genopen}. The first main result of this section is the following theorem.

\begin{te}\label{denslenstra}
Let $R$ be an order in a number field $K$ and let $A$ be an $R$-algebra which is free as an $R$-module.
Suppose that $k\geq r_{\f p}$ for all prime ideals $\f p$ of $R$ and $k\geq 1+r_0$. Then $\den_k(A)>0$.
In particular, $k\geq r$.
\end{te}

Our proof of Theorem \ref{denslenstra} will use the following
nice result, often called Loomis-Whitney inequality (\cite{whitney})

\begin{nle}\label{whitney}
Let $T$ be a finite set, $D$ a subset of $T^s$, and let $D_i$ be the projection of $D$ to $T^{s-1}$ along
the $i$-th coordinate. Then $|D|^{s-1}\leq \prod_{i=1}^{s} |D_i|$.
\end{nle}

As a corollary we get the following lemma.

\begin{nle}\label{estimate} Let $\mathbb F$ be a finite field with $q$ elements.
Suppose that an $m$-dimensional $\mathbb F$-algebra $A$ can be generated by $r$ elements as an
$\mathbb F$-algebra. For $k\geq r$ let
$\ng_k$ be the number of $k$-tuples in $A^k$ which do not generate  $A$ as an $\mathbb F$-algebra.
Then
$\text{\rm ng}_k\leq m^{2k/r}q^{mk-k/r}$ for any $k>r$.
\end{nle}

\begin{proof} Let $D(k)\subseteq A^k$
be the set of all $k$-tuples which do not generate $A$ as an $\mathbb F$-algebra.
For each $i$ the projection
$D(k)_i\subseteq A^{k-1}$ of $D(k)$ along the $i$-th coordinate is contained in $D(k-1)$.
By the Loomis-Whitney inequality (Lemma \ref{whitney}) we have
\[ \ng_k^{k-1}\leq \ng_{k-1}^k.\]
A straightforward induction yields now the inequality
\[\ng_k\leq \ng_{r}^{k/r}.\]
The set $D(r)$
is contained in the set of all zeros of some non-zero polynomial of degree $\leq m^2$ in $rm$ variables
by Lemma \ref{polynomialsgen}. It follows that $|D(r)|=\ng_r\leq m^2q^{mr-1}$ by Lemma \ref{zeros}.
Consequently,
\[ \ng_k\leq \ng_r^{k/r} \leq m^{2k/r} q^{km-k/r}.\]
\end{proof}

\begin{proof}[{\rm \bf Proof of Theorem \ref{denslenstra}}]
Recall that by Theorem \ref{gendensity} we have
\[ \den_k(A)=\prod_{\f p\in \mspec R}\frac{g_k(\f p, A)}{\N(\f p)^{mk}}.\]
Since $k\geq r_{\f p}$, we see that $g_k(\f p,A)>0$
for all maximal ideals  $\f p$. It suffices therefore to show that
$\displaystyle \prod \frac{g_k(\f p, A)}{\N(\f p)^{mk}}>0$, where
the product is over all maximal ideals with sufficiently large norm.
Since the set of all prime ideals $\f p$ of $R$ such that
$r_{\f p}=r_0$ is open and contains the zero ideal, we have $r_{\f p}=r_0$ for all but a finite number
of maximal ideals $\f p$. Let $m$ be the rank of the free $R$-module $A$.
Since $k\geq r_0+1$, Lemma \ref{estimate} implies that
$g_k(\f p,A)\geq \N(\f p)^{km}- m^{2k/r_0}\N(\f p)^{km-k/r_0}$ for every
$\f p\in \mspec R$ such that $r_0=r_{\f p}$.
It follows that
\[\frac{g_k(\f p, A)}{\N(\f p)^{mk}}\geq 1-\frac{m^{2k/r_0}}{\N(\f p)^{k/r_0}}\]
for all but a finite number of maximal ideals $\f p$.
It suffices therefore to show that $\displaystyle \prod \left(1-\frac{m^{2k/r_0}}{\N(\f p)^{k/r_0}}\right)>0$, where
the product is over all maximal ideals with sufficiently large norm. This in turn is equivalent to
showing that the series $\displaystyle \sum_{\f p\in\mspec R} \frac{m^{2k/r_0}}{\N(\f p)^{k/r_0}}$ converges, which is
indeed true since $k/r_0>1$.
\end{proof}

As an immediate corollary of Theorem \ref{denslenstra} we get the following

\begin{te}\label{lenstra}
Let $R$ be an order in a number field $K$ and let $A$ be an $R$-algebra which is free as an $R$-module.
If $r_0< r_{\f p}$ for some maximal ideal $\f p$ of $R$ then $r=\max\{ r_{\f p}: \f p\in \mspec R\}$.
If $r_0=r_{\f p}$ for all maximal ideals $\f p$ then $r_0\leq r\leq 1+r_0$.
\end{te}

A special case of Theorem \ref{lenstra}, when $R=\mathbb Z$ was communicated to us by
H.W. Lenstra \cite{lenstra}. Lenstra's proof of this result is purely algebraic and it does not
provide any way to handle the
ambiguity for $r$ when $r_0=r_{\f p}$ for all maximal ideals $\f p$. It is known that in this case
both $r=r_0$ and $r=r_0+1$ is possible. For example, there are infinitely many number fields
in which the ring of integers $A$ considered as a $\mathbb Z$-algebra has $r_{\f p}=1$ for all prime
ideals $\f p$ but $r=2$. As an explicit example one can take the ring of integers in the
cubic field $\mathbb Q(\sqrt[3]{198})$ (\cite[page 167]{pleasant}).
Later, we will see examples where
$\den_{r_0}(A)>0$, hence $r=r_0$, even though we are unable to find generators.

\begin{question}
{\rm Let $R$ be an order in a number field. Suppose that $A$ is an $R$-algebra which is finitely generated
and projective as an $R$-module. The right hand side of the formula in Theorem \ref{gendensity} makes perfect
sense for $A$ and we will continue to denote it by $\den_k A$. Is it true that if $\den_k A>0$ then
$A$ can be generated by $k$ elements as an $R$-algebra ? We believe that the answer is positive.
Perhaps there is a notion of density in this case which makes Theorem \ref{gendensity} valid?
}
\end{question}

We have the following generalization of the original result of Lenstra.

\begin{te}\label{swan}
Let $R$ be a commutative ring of dimension $\leq 1$ such that $\mspec R$ is Noetherian and let $A$ be an
$R$-algebra finitely generated as an $R$-module. Let $h$ be the smallest non-negative integer such that
$h\geq r_{\f p}$ for all but a finite number of maximal ideals $\f p$ of $R$.
Suppose that $k\geq r_{\f p}$ for all maximal ideals $\f p$ and $k\geq 1+h$.
Then $A$ can be generated by $k$ elements as an $R$-algebra.
\end{te}

\begin{proof}
Since $\mspec R$ is Noetherian, it has a finite number of irreducible components. Note that if an irreducible
component
of $\mspec R$ is finite then it consists of a single maximal ideal. Otherwise it contains infinitely
many maximal ideals and the intersection of all these ideals is a prime ideal which we call the generic ideal
of the component. Let $T$ be the set of all prime ideals which are generic ideals of some infinite irreducible
component of $\mspec R$. Thus $T$ is a finite set of minimal prime ideals of $R$ (it can be empty).
Note that if $\f p\in T$
then $r_{\f p}\leq r_{\f q}$ for any maximal ideal $\f q$ containing $\f p$ and the equality holds
for all but a finite number of such maximal ideals by Corollary \ref{genopen}.
It follows that  $h=\max\{r_{\f p}:\f p\in T\}$.
For each prime $\f p \in T$ choose a maximal ideal $\f q\supseteq \f p$ such that $r_{\f p}=r_{\f q}$
and denote this set of chosen maximal ideals by $M$.

We call a sequence $a_1,\ldots,a_m$ of elements of $A$ {\bf $M$-generic} if it generates $A$ at $\f p$
for every $\f p\in M$. Note that an $M$-generic sequence generates $A$ at $\f p$
for all but a finite number of maximal ideals $\f p$.
We claim that there is an $M$-generic sequence of length $h$.
Indeed, for each $\f q \in M$ there are h elements in $A$ which generate $A$ at $\f q$.
By the Chinese Remainder Theorem for modules, we may find elements
$a_1,\ldots,a_h$ in $A$ which generate $A$ at $\f q$ for all $\f q\in M$. Thus $a_1,\ldots,a_h$ is $M$-generic.

We will now show that for every $i\leq h$ there is an $M$-generic sequence $b_1,\ldots,b_h$ such that
for every maximal ideal $\f q$ the elements $b_1,\ldots,b_i$ can be
completed to a set of $k$ elements which generate $A$ at $\f q$. Our argument is by induction on $i$. It is
clearly true for $i=0$
(any $M$-generic sequence of length $h$ works). Suppose that $b_1,\ldots,b_h$ is a generic sequence
which works for some $i$. We seek a generic sequence working for $i+1$ which is of the form
$b_1,\ldots,b_i, b,b_{i+2},\ldots,b_h$ for some $b\in A$. Note that if $b$ is such that $b-b_{i+1}\in \f q A$
for all $\f q\in M$ then $b_1,\ldots,b_i, b,b_{i+2},\ldots,b_h$ is $M$-generic. Also, there is a finite
set $W$ of maximal ideals, disjoint from $M$, such that for any maximal ideal $\f q\not\in W$
and any $b\in A$, the sequence $b_1,\ldots,b_i,b,b_{i+1},\ldots,b_h$ generates $A$ at $\f q$.
Since $k>h$, for any $\f q\not \in W$ and any $b\in A$, the elements $b_1,\ldots,b_i,b$ can be
completed to a set of $k$ elements which generate $A$ at $\f q$. So in our choice of $b$
we only need to worry about maximal ideals in $W$.
For every $\f q\in W$ there is $b_{\f q}\in A$
such that $b_1,\ldots,b_i, b_{\f q}$ extends to a set of $k$ elements which generate $A$ at $\f q$.
By the Chinese Remainder Theorem for modules, we may choose $b\in A$
such that $b-b_{i+1}\in \f q A$ for all $\f q\in M$ and $b-b_{\f q}\in \f q A$ for all $\f q\in W$.
For any such $b$ the sequence $b_1,\ldots,b_i, b,b_{i+2},\ldots,b_h$ has the required properties for $i+1$.

Let $a_1,\ldots,a_h$ be an $M$-generic sequence good for $i=h$. Thus, for any maximal ideal $\f q$ outside
some finite set $U$ the elements $a_1,\ldots,a_h$ generate $A$ at $\f q$. For each $\f q\in U$, there
are elements $a_{h+1}(\f q),\ldots,a_{k}(\f q)$ in $A$ such that $a_1,\ldots,a_h,a_{h+1}(\f q),\ldots,a_{k}(\f q) $
generate $A$ at $\f q$. By the Chinese Remainder Theorem for modules, there are elements $a_{h+1},\ldots,a_k$
in $A$ such that $a_i-a_i(\f q)\in \f qA$ for all $\f q\in U$ and all $i=h+1,\ldots,k$. Thus
the elements $a_1,\ldots,a_k$ generate $A$ at $\f q$ for every maximal ideal $\f q$, hence they generate
$A$ as an $R$-algebra by Lemma \ref{localgen}.
\end{proof}

The reader familiar with the results of Forster and Swan on the number of generators of modules
over Noetherian commutative rings should recognize the similarities between Theorem \ref{swan}
and Swan's Theorem \cite[Theorem 5.8]{matsumura}. Unlike the result of Swan, Theorem \ref{swan}
only treats the case of rings of dimension $\leq 1$. So far we have not been able to get similar results
for rings of higher dimension but we believe that the following conjectural generalization should be true.
In order to state it we need to recall briefly some notions
(see \cite[pages 35-37]{matsumura} for more details). We denote by $\jspec R$ the subspace of $\spec R$
which consists of those prime  ideals which are intersections of some set of maximal ideals of $R$.
We assume that $\mspec R$ is a Noetherian space. It turns out that this is equivalent to $\jspec R$
being Noetherian and then both spaces have the same combinatorial dimension. When $\f p \in \jspec R$,
we write $\text{j-dim~} \f p$ for the combinatorial dimension of the closure of $\{\f p\}$ in $\jspec R$.
For $\f p \in \jspec R$ define
\[ b(\f p,A)=\begin{cases} 0 & \text{if $A_{\f p}=0$} \\
\text{j-dim~} \f p +r_{\f p}( A) & \text{if $A_{\f p}\neq 0$}
\end{cases}\]

\begin{conjecture}\label{swan1}
Suppose that $R$ is a commutative ring such that {\rm $\mspec R$} is a Noetherian space. Let $A$
be an $R$-algebra finitely generated as an $R$-module. If {\rm
\[ \sup\{b(\f p,A):\f p \in \mspec R\}=n<\infty\]
}
then $A$ can be generated as an $R$-algebra by $n$ elements.
\end{conjecture}

\section{Generators of matrix algebras over finite fields}\label{s6}
It is clear from the results of Section \ref{s3} that the key step towards understanding the smallest number
of generators of an algebra over a commutative ring is to handle the case of algebras over fields.
Among the finite dimensional algebras over fields the best understood class is the class of separable
algebras. It was proved in \cite{mmbp} that any separable algebra over an infinite field is two generated.
This is no longer true over finite fields. In this case, separable algebras coincide with finite products
of matrix algebras.

By Proposition \ref{matrixgen}, understanding the structure of generators of a semisimple $F$-algebra
reduces to algebras of the form $A^m$, where $A$ is a simple $F$-algebra. We have the following result:

\begin{theorem}\label{simple}
Let $F$ be a field, let $A$ be a finite dimensional simple $F$-algebra, and let $k,m,n$ be positive integers.
Then $k$ elements of $A^m$
\[a_1=(a_{11}, \ldots,  a_{1m}), \ldots  ,a_k=(a_{k1}, \ldots,  a_{km})\]
generate $A^m$ as an $F$-algebra if and only if the following
two conditions are satisfied:
\begin{enumerate}[\rm (1)]
\item\label{it3} For any $i = 1, \ldots, m$, the elements $a_{1i}, \ldots, a_{ki}$ generate $A$ as an
$F$-algebra.
\item\label{it4}  There does not exist a pair of different indices  $i,j$ for which there is an
automorphism $\Psi$ of the $F$-algebra $A$ such that
\[a_{1i} = \Psi( a_{1j}), \ldots,  a_{ki} = \Psi(a_{kj}).\]
\end{enumerate}
\end{theorem}

\begin{proof}
Let $B$ denote the subalgebra of $A^m$ generated by $a_1,\dots,a_k$.
Recall that there is a unique up to isomorphism simple $A$-module $M$ and it is faithful.
Let $M_i$ be the pull-back of $M$ via the projection $\pi_i:B\longrightarrow A$ on the $i$-th coordinate.
Thus $M_i$ is a $B$-module which coincides with $M$ as an $F$-vector space and for $b\in B$ and
$m\in M_i=M$ we have $bm=\pi_i(b)m$.
Since $\pi_i$ is surjective by (\ref{it3}), each $M_i$ is a simple $B$ module. We claim that these $B$-modules
are pairwise non-isomorphic. Indeed,
suppose that for some
$i\neq j$ the $B$-modules $M_i$ and $M_j$ are isomorphic and let $\Phi: M_i\longrightarrow M_j$ be an
isomorphism of these $B$-modules. For any $a\in A$ there is $b\in B$ such that $\pi_i(b)=a$. Set
$\Psi(a)=\pi_j(b)$. We claim that $\Psi$ is well defined and it is an automorphism of the $F$-algebra $A$.
Indeed, if $b_1\in B$ is another element such that $\pi_i(b_1)=a$ then for any $m\in M_i$ we have
$bm=b_1m$. Applying $\Phi$ to this equality, we see that $b\Phi(m)=b_1\Phi(m)$ for any $m\in M_i$.
Since $\Phi$ is an isomorphism, we conclude that $bn=b_1n$ for any $n\in M_j$, i.e. $\pi_j(b)m=
\pi_j(b_1)m$ for every $m\in M$.
Since $M$ is a faithful $A$-module, we conclude that $\pi_j(b)=\pi_j(b_1)$. This shows that $\Psi$ is
well defined. It is now straightforward to see that $\Psi$ respects addition and multiplication and it is
$F$-linear. It follows that $\Psi$ is an isomorphism of $F$-algebras. This however is in contradiction
with our assumption (\ref{it4}). It follows that $M_i$ and $M_j$ are not isomorphic as $B$-modules for
$i\neq j$.
Note that $\oplus_{i=1}^{m}M_i$ is a semisimple, faithful $B$-module. It follows that $B$ is
semisimple and every simple $B$-module is isomorphic to one of the $M_i$'s. By the Wedderburn-Artin theory,
$B$ is isomorphic to the product $\prod_{i=1}^m B_i$, where $B_i=\M_{n_i}(D_i)$, $D_i=\End_B(M_i)$, and
$n_i\dim_F(D_i)=\dim_F(M_i)=\dim_FM$. Note that $D_i=\End_B(M_i)=\End_A(M)$ and therefore
$A$ is isomorphic to $B_i$ for each $i$, again by the Wedderburn-Artin theory. This proves that $\dim_F A^m=
\dim_F B$, and consequently $A^m=B$.
\end{proof}

As a simple corollary we get the following

\begin{pro}\label{orbits}
Let $A$ be a simple finite dimensional algebra over a field $F$. For any $k>0$ the group $\Aut_F(A)$
of $F$-algebra automorphisms of $A$ acts freely on the set $\Gen_k(A,F)$. The algebra $A^m$ can be generated
by $k$ elements as an $F$-algebra if and only if there are at lest $m$ different orbits of the action of $\Aut_F(A)$
on $\Gen_k(A,F)$.
\end{pro}
\begin{proof}
The action of $\Aut_F(A)$ on $\Gen_k(A,F)$ is the restriction of the coordinatewise action of $\Aut_F(A)$
on $A^k$. If $\Psi\in \Aut_F(A)$ fixes an element of $\Gen_k(A,F)$, then it fixes each member of a set of
generators of $A$ as an $F$-algebra, so $\Psi$ is the identity. This explains why the action is free.
Theorem \ref{simple} says that elements $a_1=(a_{11}, \ldots,  a_{1m}), \ldots  $, $
a_k=(a_{k1}, \ldots,  a_{km})$ generate $A^m$ as an $F$-algebra iff the elements $(a_{11}, \ldots,  a_{k1}),
\ldots$, $(a_{1m}, \ldots,  a_{km})$ belong to different orbits of the action of $\Aut_F(A)$
on $\Gen_k(A,F)$.
\end{proof}

Suppose now that $F=\mathbb F_q$ is a finite field with $q$ elements. Then simple finite dimensional
$\mathbb F_q$-algebras are exactly algebras of the form $\M_n(\mathbb F_{q^s})$ for some positive integers $n,s$.
Now, by the Skolem-Noether Theorem, the group of automorphisms of the $\mathbb F_q$-algebra
$\M_n(\mathbb F_{q^s})$ is the semidirect product of the group $\text{PGL}_n(\mathbb F_{q^s})$ and
the Galois group $\text{Gal}(\mathbb F_{q^s}/\mathbb F_{q}$). Thus we get the following

\begin{theorem}\label{formula}
Let $A=\M_n(\mathbb F_{q^s})$ considered as an $\mathbb F_{q}$-algebra. Then $A^m$ can be generated
by $k$ elements as an $\mathbb F_{q}$-algebra iff
\[ m\leq \frac{|\Gen_k(A,\mathbb F_{q})|}{s|\text{\rm PGL}_n(\mathbb F_{q^s})|}.\]
Furthermore,
\[ |\Gen_k(A^m,\mathbb F_{q})|=\prod_{i=0}^{m-1}(|\Gen_k(A,\mathbb F_{q})|-i\cdot s\cdot |\text{\rm PGL}_n(\mathbb F_{q^s})|)\]
\end{theorem}
\begin{proof}
As we noted above, $\Aut_{\mathbb F_q}(A)$ has $s|\text{PGL}_n(\mathbb F_{q^s})|$ elements.
Since $\Aut_{\mathbb F_q}(A)$ acts freely on $\Gen_k(A,\mathbb F_{q})$, the number of orbits
of this action is equal to
$\displaystyle \frac{|\Gen_k(A,\mathbb F_{q})|}{s|\text{\rm PGL}_n(\mathbb F_{q^s})|}$.
The first part of the theorem is now an immediate consequence of Proposition \ref{orbits}.

To prove the second part note that according to Proposition \ref{orbits} the elements
of $\Gen_k(A^m,\mathbb F_{q})$ are in bijective correspondence with sequences of length $m$
of elements from $\Gen_k(A,\mathbb F_{q})$, with no two elements in the same orbit of
$\Aut_{\mathbb F_q}(A)$. In order to count these sequences, let $o$ be the number of orbits
of the action of $\Aut_{\mathbb F_q}(A)$ on $\Gen_k(A,\mathbb F_{q})$
and let $t$ be the size of each orbit.
We can choose a sequence of $m$ different orbits $O_1,\ldots,O_m$ in $m!{o\choose m}$
ways and the number of sequences $g_1,\ldots,g_m$ such that $g_i\in O_i$ for $i=1,\ldots,m$
is $t^m$. Thus
\[|\Gen_k(A^m,\mathbb F_{q})|=m!{o\choose m}t^m=\prod_{i=0}^{m-1}(ot-it).\]
 The second part of the theorem follows now immediately from the equalities
$ot=|\Gen_k(A,\mathbb F_{q})|$ and
$t=s|\text{\rm PGL}_n(\mathbb F_{q^s})|$.
\end{proof}

Perhaps it is worth to mention that for a simple separable algebra $A$ over any field $F$
the sets $\Gen_k(A,F)$ are non-empty for any $k\geq 2$. In other words, we have the following

\begin{te}
Let $A$ be a simple separable algebra over a field $F$. Then $A$ can be generated by two
elements as an $F$-algebra.
\end{te}
\begin{proof}
For infinite fields $F$ the result has been proved in \cite{mmbp}. When $F=\mathbb F_{q}$
is a finite field with $q$ elements then $A$ is isomorphic to $\M_n(\mathbb F_{q^s})$
for some positive integers $n$ and $s$. Let $u$ be a generator of the multiplicative group
of $\mathbb F_{q^s}$, so in particular $\mathbb F_{q^s}=\mathbb F_{q}[u]$.
For $1 \le i,j \le n$ let
$E_{ij}$ denote the matrix whose $(i,j)$ entry is $1$
and all other entries are $0$. Let $A = uE_{11}$
and $B= E_{1n} + \sum_{i=1}^{n-1} E_{i+1,i\,}$. Then
$u^kE_{ij} = B^{i-1}A^kB^{n+1-j}$ for all $1\leq i,j\leq n$ and all $k\geq 0$. It follows that $A$ and $B$
generate the $\mathbb F_{q}$-algebra $\M_n(\mathbb F_{q^s})$.

\end{proof}

\section{The numbers $|\Gen_k(\M_n(\mathbb F_q),\mathbb F_{q})|$}\label{s7}
In this section we will attempt to get some information about the numbers
$|\Gen_k(\M_n(\mathbb F_q),\mathbb F_{q})|$. In particular, we will compute them when $n\leq 3$.
To simplify the notation, we make the following definition.

\begin{definition} {\rm Let $m,n$ be positive integers and let $q$ be a prime power.
We introduce the following notation:
\begin{enumerate}[\rm (i)]
\item $\G_{m,n}(\mathbb F_q)=\Gen_m(\M_n(\mathbb F_q),\mathbb F_{q})$.

\item $\g_{m,n}(q)=|\G_{m,n}(\mathbb F_q)|$.

\item $\displaystyle \gen_{m,n}(q)=\frac{\g_{m,n}(q)}{|\text{\rm PGL}_n(\mathbb F_{q})| }$.
\end{enumerate}
}
\end{definition}
Note that by Theorem~\ref{formula}, the number $\gen_{m,n}(q)$ is equal to
the largest $k\in \mathbb Z$ such that $r(\M_n(\mathbb F_q)^k,\mathbb F_{q})\leq m$. Thus our notation
agrees with that introduced in Definition~\ref{defgen}.

When $n=1$, an $m$-tuple generates $\mathbb F_{q}$ if and only if it contains a non-zero element.
It follows that $\g_{m,1}(q)=q^m-1$. From now on in this section we assume that $n\geq 2$, unless
stated otherwise.

Our attempt at computing the numbers $\g_{m,n}(q)$ is based on the following simple observation: a set
of matrices does not generate the whole algebra $\mt{n}$ if and only if there is a maximal subalgebra
of $\mt{n}$ that contains this set. Thus the following is true

\begin{eqnarray}\label{avoidmax}
\G_{m,n}\left( \mathbb{F}_q \right) =\ \ \ \ \ \ \ \ \ \ \ \ \ \ \ \ \ \ \ \ \ \ \ \ \ \ \ \ \ \ \ \ \ \\ \nonumber
\M_n(\mathbb{F}_q)^m - \bigcup \left\{ \mathcal{A}^m :
~\mathcal{A}~\text{is a maximal subalgebra of}~\M_n(\mathbb{F}_q) \right\}\text{.}
\end{eqnarray}
Let $\Dg$ be the  subalgebra of scalar matrices of  $\mt{n}$. Since any subalgebra of $\mt{n}$ contains
${\Dg}$, we can subtract ${\Dg}^m$ in the above formula and get that $G_{m,n}(\Fg_q)$ is equal to
\[
\M_n(\mathbb{F}_q)^m -{\Dg}^m - \bigcup \left\{ \mathcal{A}^m -{\Dg}^m :
~\mathcal{A}~\text{is a maximal subalgebra of}~\M_n(\mathbb{F}_q) \right\}\text{.}
\]
Since $|\mt{n}|=q^{n^2}$ and $|\Dg|=q$, the inclusion-exclusion formula yields
$$\g_{m,n}(q)=q^{mn^2}-q^m+\sum(-1)^{k}|({\Ag}_{i_1}^m-{\Dg}^m)\cap\ldots\cap ({\Ag}_{i_k}^m-{\Dg}^m)|,$$
where the sum is taken over all non-empty subsets $\{{\Ag}_{i_1},\ldots,{\Ag}_{i_k}\}$ of the set of all maximal
subalgebras of $\M_n(\Fg_q)$. Since $\Dg$ is contained in every subalgebra of $\mt{n}$, we have
\[({\Ag}_{i_1}^m-{\Dg}^m)\cap\ldots\cap ({\Ag}_{i_k}^m-{\Dg}^m)={\Ag}_{i_1}^m\cap\ldots\cap{\Ag}_{i_k}^m-{\Dg}^m=
({\Ag}_{i_1}\cap\ldots\cap{\Ag}_{i_k})^m-{\Dg}^m
\]
and therefore
\begin{equation}\label{Counting}
\g_{m,n}(q)=q^{mn^2}-q^m+\sum(-1)^{k}\left(|{\Ag}_{i_1}\cap\ldots\cap{\Ag}_{i_k}\right)|^m-q^m),
\end{equation}
where the sum is taken over all non-empty subsets $\{{\Ag}_{i_1},\ldots,{\Ag}_{i_k}\}$ of the set of all maximal
subalgebras of $\M_n(\Fg_q)$.

In order to evaluate the right hand side of (\ref{Counting}), it is necessary to have a description of
all maximal subalgebras of $\mt{n}$.
It is quite easy to produce one type of maximal subalgebras of $\mt{n}$. In fact, we have the following
result.

\begin{lemma}\label{stabmax}
 For a proper non-trivial vector subspace $U$ of $\Fg_q^n$ let ${\Ag}_U$ be the set of all matrices from $\mt{n}$ that
 leave $U$ invariant. Then ${\Ag}_U$ is a maximal subalgebra of $\mt{n}$.
 Moreover, each ${\Ag}_U$ is uniquely determined by $U$, that is, if ${\Ag}_U={\Ag}_{U'}$ then $U=U'$.
\end{lemma}
\begin{proof}
First note that the center of ${\Ag}_U$ consists of scalar matrices. In fact, if a matrix $A$ is in the center
of ${\Ag}_U$ then it acts as a scalar $\lambda$ on $U$. The matrix $B=A-\lambda I$ annihilates $U$
and is in the center of ${\Ag}_U$. Suppose that $Bv\neq 0$ for some $v$. Then there is a projection $\Pi$
onto $U$ such that $\Pi(Bv)\neq 0$. Since $\Pi \in {\Ag}_U$, we have $0\neq \Pi Bv=B\Pi v=0$, a contradiction.
Thus $B=0$ and $A$ is a scalar matrix.

Now note that if $U\neq U'$ then there is $A\in {\Ag}_U-{\Ag}_{U'}$. In fact, if $U'\subsetneq  U$ then
such $A$ clearly exists since ${\Ag}_U$ is transitive on $U$. If there is $v\in U'-U$ then for any $w$
there is $A\in {\Ag}_U$ such that $Av=w$. Taking $w\not\in U'$ yields required $A$. This, in particular,
proves the second assertion.

Take any matrix $A$ not in ${\Ag}_U$ and let $\Ag'$ be the algebra generated by $A$ and ${\Ag}_U$.
Note that $\Ag'$ can not fix any non-trivial subspace $V$ of $\Fg_q^n$. In fact, if $V\neq U$ then,
as we have seen above, ${\Ag}_U$ is not contained in ${\Ag}_V$ and $A$ does not take $U$ into $U$. Thus $\Fg_q^n$
is a simple and faithful $\Ag'$-module. It follows that $\Ag'$ is a simple central $\Fg_q$-algebra with
a simple module of dimension $n$ over $\Fg_q$, hence it must be isomorphic to $\mt{n}$.
It follows that ${\Ag}_U$ is maximal.

\end{proof}
The following lemma describes a second type of maximal subalgebras of $\mt{n}$.

\begin{nle}\label{fieldmax}
Let $s$ be a prime divisor of $n$ and let $m=n/s$. Any $\Fg_q$-subalgebra of $\mt{n}$
isomorphic to $\mx{m}{s}$ is maximal. Any two such subalgebras are conjugate in $\mt{n}$ and their
number is equal to $\displaystyle s^{-1}\prod_{s\nmid i, 1\leq i<n}(q^n-q^i)$.
\end{nle}
\begin{proof}
Let ${\Ag}$ be a $\Fg_q$-subalgebra of $\mt{n}$ isomorphic to $\mx{m}{s}$.
Thus $\Fg_q^n$ is an ${\Ag}$-module of dimension $m$ over the center of ${\Ag}$ (which
is isomorphic to $\Fg_{q^s}$). It follows that $\Fg_q^n$ is a simple ${\Ag}$-module.
Suppose that ${\Ag}'$ is a $\Fg_q$-subalgebra of $\mt{n}$ containing ${\Ag}$.
Then $\Fg_q^n$ is a simple and faithful ${\Ag}'$-module. It follows that ${\Ag}'$
is simple, hence it is isomorphic to $\mx{k}{r}$, where $kr=n$ and $r$ is the dimension
of the center of ${\Ag}'$ over $\Fg_q$. Clearly, the center of ${\Ag}'$ is contained
in the center of ${\Ag}$. It follows that $r|s$, and therefore $r=1$ or $r=s$ (recall that
$s$ is a prime). In the former case we get ${\Ag}'=\mt{n}$ and in the latter case
we have ${\Ag}'={\Ag}$. This shows that ${\Ag}$ is maximal.

For the existence of an $\Fg_q$-subalgebra of $\mt{n}$
isomorphic to $\mx{m}{s}$ consider the (unique up to isomorphism) simple  $\mx{m}{s}$-module
$V$. It has dimension $m$ as a vector space over $\Fg_{q^s}$, so as a $\Fg_{q}$-vector space
it is isomorphic to $\Fg_q^n$. Thus the action of $\mx{m}{s}$ on $V$ induces
an $\Fg_{q}$-algebra embedding of $\mx{m}{s}$ into $\mt{n}$.

Fix now a $\Fg_q$-subalgebra ${\Ag}$ of $\mt{n}$ isomorphic to $\mx{m}{s}$. By the Noether-Skolem theorem,
any $\Fg_q$-algebra homomorphism of ${\Ag}$ into $\mt{n}$ is given by conjugation with some invertible
element of $\mt{n}$. This means that the group $\text{\rm GL}_n(\Fg_q)$ acts transitively on the set of
subalgebras of $\mt{n}$ which are isomorphic to $\mx{m}{s}$. Since ${\Ag}$ is
maximal, the subgroup $C$ of elements which act trivially on ${\Ag}$ coincides with the multiplicative group of
the center of ${\Ag}$. The quotient of the stabilizer of ${\Ag}$ by $C$ is,
again by the  Noether-Skolem theorem, isomorphic to the group of all automorphism of the
$\Fg_q$-subalgebra ${\Ag}$. We have seen earlier that the group of $\Fg_q$-algebra automorphisms of  $\mx{m}{s}$
has $s|\text{\rm PGL}_m(\mathbb F_{q^s})|$ elements (see the discussion directly before Theorem~\ref{formula}).
Therefore the stabilizer
of ${\Ag}$ has $|C|\cdot s\cdot |\text{\rm PGL}_m(\mathbb F_{q^s})|=s|\text{\rm GL}_m(\mathbb F_{q^s})|$
elements. Consequently, the number of
$\Fg_q$-subalgebra ${\Ag}$ of $\mt{n}$ isomorphic to $\mx{m}{s}$ is equal to
\[\frac{|\text{\rm GL}_n(\Fg_q)|}{s|\text{\rm GL}_m(\mathbb F_{q^s})|}=
s^{-1}\prod_{s\nmid i, 1\leq i<n}(q^n-q^i).\]
\end{proof}
It turns out that the maximal subalgebras described in Lemmas \ref{stabmax} and \ref{fieldmax}
exhaust all possible maximal subalgebras. In other words, we have the following result.

\begin{proposition}\label{clasmax}
Let ${\Ag}$ be a maximal $\Fg_q$-subalgebra of $\mt{n}$. Then either ${\Ag}={\Ag}_U$ for some subspace $U$ of
$\Fg_q^n$ or ${\Ag}$ is isomorphic to $\mx{m}{s}$ for some prime divisor $s$ of $n=ms$.
\end{proposition}
\begin{proof}
Suppose that ${\Ag}$ fixes some proper non-trivial subspace $U$ of $\Fg_q^n$. Then ${\Ag}$ is contained in ${\Ag}_U$, hence
${\Ag}={\Ag}_U$. If no proper non-trivial subspace of  $\Fg_q^n$ is fixed by ${\Ag}$ then
$\Fg_q^n$ is a simple and faithful ${\Ag}$-module. It follows that ${\Ag}$ is simple and therefore
it is isomorphic to $\mx{k}{r}$, where $kr=n$. Let $s$ be a prime divisor of $r$. The center of
${\Ag}$ contains a subfield $F$ isomorphic to $\Fg_{q^s}$. The centralizer of $F$ in $\mt{n}$
consists exactly of those linear transformations of $\Fg_q^n$ which are $F$-linear. Thus it is
a subalgebra of $\mt{n}$ isomorphic to $\mx{m}{s}$, where $ms=n$. On the other hand, this subalgebra
contains ${\Ag}$, hence it must be equal to ${\Ag}$.
\end{proof}

In order to carry our strategy to compute the numbers $\g_{m,n}(q)$ we need to understand the intersections
of maximal subalgebras of $\mt{n}$. This appears to be a very challenging combinatorial problem and so far
we have only succeeded to complete the computations for $n\leq 3$. One of the complications in the general case
is that the maximal subalgebras are of two different types. This difficulty disappears when $n$ is a prime
by the following observation.

\begin{lemma}\label{lm3}
Let $n$ be a prime number. If $\Ag$ is a maximal subalgebra of $\mt{n}$ isomorphic to $\Fg_{q^n}$, then its
intersection with any other maximal subalgebra is equal to $\Dg$, the algebra of scalar matrices.
\end{lemma}
\begin{proof}
Since $n$ is prime, $\Fg_{q^n}$ has only two subfields, itself and $\Fg_q$. In other words,
$\Ag$ has only two subalgebras, $\Ag$ and $\Dg$. Since the intersection cannot be equal to $\Ag$, it is equal to $\Dg$.
\end{proof}

For the rest of this section we assume that $n$ is a prime number. Thus
Lemma \ref{lm3} tells us that if the set $\{{\Ag}_{i_1},\ldots,{\Ag}_{i_k}\}$ of maximal subalgebras of $\mt{n}$
includes a subalgebra isomorphic to $\Fg_{q^n}$, and $k\geq 2$, then the intersection of the subalgebras in this
set is equal to $\Dg$, and so the corresponding term in (\ref{Counting}),
$|{\Ag}_{i_1}\cap\ldots\cap{\Ag}_{i_k}|^m-q^m$, is equal to $0$. It follows that we can rewrite (\ref{Counting})
in the following way:
\[\g_{m,n}(q)=q^{mn^2}-q^m-
\sum_{{\Ag}\cong\Fg_{q^n}}\left(|{\Ag}|^m-q^m\right)+\]
\[+\sum(-1)^{k}\left(|{\Ag}_{U_1}\cap\ldots
\cap{\Ag}_{U_k}|^m-q^m\right)\]
where the second sum is over all non-empty sets $\{U_1,\ldots,U_k\}$ of non-trivial proper subspaces of $\Fg_q^n$.
By Lemma \ref{fieldmax}, the first sum consists of $n^{-1}\prod_{i=1}^{n-1}(q^n-q^i)$ terms, each term being
$q^{mn}-q^m$. Thus we get the following formula
\begin{eqnarray}\label{preparation0}
\g_{m,n}(q)=q^{mn^2}-q^m-n^{-1}(q^{mn}-q^m)\prod_{i=1}^{n-1}(q^n-q^i)+ \\ \nonumber  +\sum(-1)^{k}\left(
|{\Ag}_{U_1}\cap\ldots \cap{\Ag}_{U_k}|^m-q^m\right).
\end{eqnarray}
where the sum is over all non-empty sets $\{U_1,\ldots,U_k\}$ of non-trivial proper subspaces of $\Fg_q^n$.

Let $\fg$ be the set of all subalgebras of $\mt{n}$ which are intersections of some of the maximal
algebras of the form ${\Ag}_{U}$. For each  ${\Ag}\in \fg$, define the degree $\d({\Ag})$ of ${\Ag}$ by
\begin{equation}
\d({\Ag})=\sum (-1)^k,
\end{equation}
where the sum is over all sets $\{U_1,\ldots,U_k\}$ of non-trivial proper subspaces of $\Fg_q^n$ such that
${\Ag}_{U_1}\cap\ldots \cap{\Ag}_{U_k}={\Ag}$. Thus (\ref{preparation0}) can be stated as
\begin{eqnarray}\label{preparation1}
\g_{m,n}(q)=q^{mn^2}-q^m-n^{-1}(q^{mn}-q^m)\prod_{i=1}^{n-1}(q^n-q^i)+\\ \nonumber  +
\sum_{{\Ag}\in \fg}\d({\Ag})\left(|{\Ag}|^m
- q^m\right).
\end{eqnarray}

The following simple lemma will be useful for our analysis of elements of $\fg$.

\begin{lemma}\label{scaling}
Let $F$ be a field, $V$ be a vector space over $F$, and let $v_1,\ldots,v_k \in V$ be a minimal linearly
dependent collection of vectors (so any $k-1$ of them are linearly independent). Then any linear endomorphism
of $V$ that scales $v_1,\ldots,v_k$ is a scalar operator when restricted to the linear span of
$v_1,\ldots,v_k$.
\end{lemma}
\begin{proof}
Let $f$ be a linear endomorphism of $V$ such that $f(v_i) = \alpha_i v_i$ for some $\alpha_i \in F$ and
$i = 1,\ldots,k$.
%Suppose without loss of generality that $\alpha_k\neq 0$.
The assumptions of the lemma imply that  $v_k = \beta_1 v_1 +\ldots+\beta_{k-1} v_{k-1}$ for some non-zero
$\beta_1,\ldots,\beta_{k-1} \in F$.
By expressing $f(v_k)$
in two ways, as $\beta_1 \alpha_1 v_1 +\ldots+\beta_{k-1} \alpha_{k-1}v_{k-1}$ and as $\alpha_k v_k$,
we obtain $\beta_i \alpha_i = \beta_i \alpha_k$ for all $i$. Since none of the $\beta_i$'s is $0$,
we have $\alpha_i = \alpha_k$ for $i=1,\ldots,k$.
\end{proof}

\subsection{The case $n=2$}
In this subsection we evaluate (\ref{preparation1}) in the case $n=2$. Any element ${\Ag}\in \fg$
is of the form ${\Ag}_{U_1}\cap\ldots \cap{\Ag}_{U_k}$, where $k\geq 1$ and $U_1,\ldots, U_k$
are distinct lines in $\Fg_q^2$. Note that by Lemma \ref{scaling}, ${\Ag}=\Dg$ if $k\geq 3$ and in this case
$\Ag$ does not contribute anything to (\ref{preparation1}). It follows that if ${\Ag}$ is
an element of $\fg$ different from $\Dg$, then it can be
expressed as the intersection of maximal subalgebras in a unique way and it is either of the form
$\Ag_U$ or of the form ${\Ag}_{U_1}\cap{\Ag}_{U_2}$. In the former case, we have $|{\Ag}|=q^3$
and $\d({\Ag})=-1$. In the latter case, $|{\Ag}|=q^2$ and $\d({\Ag})=1$.
%In particular, $d({\Ag})=(-1)^k$. Finally, if $k=1$ then
%${\Ag}={\Ag}_{U}$ for some line $U$ and therefore $|{\Ag}|=q^3$. If $k=2$, then ${\Ag}=
%{\Ag}_{U_1}\cap{\Ag}_{U_2}$ and $|{\Ag}|=q^2$.
Since the number of lines in $\Fg_q^2$
is $q+1$, the formula (\ref{preparation1}) takes the following form
\[\g_{m,2}(q)=q^{4m}-q^m-2^{-1}(q^{2m}-q^m)(q^2-q)-(q+1)(q^{3m}-q^m)+\]
\[+2^{-1}(q+1)q(q^{2m}-q^m),\]
which simplifies to
\begin{equation}\label{n=2}
\g_{m,2}(q)=q^{2m+1}(q^{m-1}-1)(q^m-1).
\end{equation}

\subsection{The case $n=3$}
In this subsection we evaluate (\ref{preparation1}) for $n=3$. This is substantially more difficult than the
case $n=2$, but we are still
able to analyze all elements of $\fg$. The following combinatorial lemma will help us evaluate the degree of
some of the algebras in $\fg$.

\begin{lemma}\label{degree0}
Let $X$ be a finite set. Consider a family $\mathcal S$ of subsets of $X$
such that if $Y\in{\mathcal S}$ and $Y\subseteq Y'\subseteq X$, then $Y'$ also
belongs to $\mathcal S$. Suppose furthermore that one of the following two conditions is true.
\begin{enumerate}[\rm (1)]
 \item There is $x\in X$ such that $X'-\{x\}\in {\mathcal S}$ for any $X'\in{\mathcal S}$.
 \item There are $x,y\in X$ such that if $X'\in{\mathcal S}$ and $X'-\{x\}\not\in {\mathcal S}$ then
\begin{enumerate}[\rm (a)]
%{\rm a)}
\item  $X'-\{y\}\in{\mathcal S}$ and
%and {\rm b)}
\item $(X'\cup\{y\})-\{x\}\not\in {\mathcal S}.$
\end{enumerate}

\end{enumerate}
Then $\displaystyle \sum_{Y\in{\mathcal S}}(-1)^{|Y|}=0$.
\end{lemma}
\begin{proof}
Let ${\mathcal S}_0$ be the family of those subsets from $\mathcal S$ that do not contain $x$ and let
${\mathcal S}_1$ the family of those subsets that
contain $x$. The map $t:Y\mapsto Y\cup\{x\}$ is an injection from ${\mathcal S}_0$ to ${\mathcal S}_1$. Let
 ${\mathcal S}_{2}={\mathcal S}_1-t({\mathcal S}_0)$. We have
 \[\sum_{Y\in{\mathcal S}}(-1)^{|Y|}=\sum_{Y\in{\mathcal S}_0}(-1)^{|Y|}+\sum_{Y\in t({\mathcal S}_0)}(-1)^{|Y|}+\sum_{Y\in{\mathcal S}_{2}}(-1)^{|Y|}.\]
Since $|t(Y)|=1+|Y|$, the first two sums on the right annihilate each other, and so $\sum_{Y\in{\mathcal S}}(-1)^{|Y|}=
 \sum_{Y\in{\mathcal S}_{2}}(-1)^{|Y|}$.

The condition $(1)$ of the lemma exactly means that ${\mathcal S}_{2}$ is empty, hence $\sum_{Y\in{\mathcal S}}(-1)^{|Y|}=0$.
If condition $(2)$ holds, we write ${\mathcal S}_2$ as a disjoint union ${\mathcal S}_2={\mathcal S}_{20}\cup{\mathcal S}_{21}$, where
${\mathcal S}_{20}$ consists of those elements of ${\mathcal S}_2$ which do not contain $y$. By b), the map $s:Y\mapsto Y\cup\{y\}$
maps ${\mathcal S}_{20}$ into ${\mathcal S}_{21}$ and a) implies that $s$ is onto. Thus $s:{\mathcal S}_{20}\longrightarrow {\mathcal S}_{21}$
is a bijection and
\[\sum_{Y\in{\mathcal S}_{2}}(-1)^{|Y|}=\sum_{Y\in{\mathcal S}_{20}}(-1)^{|Y|}+\sum_{Y\in{\mathcal S}_{21}}(-1)^{|Y|}=\]
\[=\sum_{Y\in{\mathcal S}_{20}}\left((-1)^{|Y|}+(-1)^{|s(Y)|}\right)=0.\]
\end{proof}
We apply of Lemma \ref{degree0} as follows. Given ${\Ag}\in \fg$, the set $X=X_{\Ag}$ will consists of all proper
non-trivial subspaces of
$\Fg_q^3$ fixed by $\Ag$ and
the family $\mathcal S=\mathcal S_{\Ag}$ will consist of all subsets $\{U_1,\ldots,U_k\}$ of $X$ such that
${\Ag}_{U_1}\cap\ldots \cap{\Ag}_{U_k}={\Ag}$.
If $(1)$ or $(2)$ holds for $\mathcal S_{\Ag}$, then Lemma \ref{degree0} tells us that $\d({\Ag})=0$.

Before we start the analysis of elements in $\fg$ let us recall that the dot product
$v\cdot w=v_1w_1+v_2w_2+v_2w_3$ is a non-degenerate symmetric bilinear form on $\Fg_q^3$.
The adjoint operator with respect to this bilinear form is the transposition. It follows that
if ${\Ag}_{U_1}\cap\ldots \cap{\Ag}_{U_k}={\Ag}\in \fg$ then
\[{\Ag}_{U_1^{\perp}}\cap\ldots \cap{\Ag}_{U_k^\perp}={\Ag^t}:=\{A^t:A\in {\Ag}\} \in \fg,\]
where $A^t$ is the transpose of $A$ and $U^{\perp}$ is the subspace orthogonal to $U$ with respect
to the dot product. We will often call ${\Ag^t}$ the dual of ${\Ag}$.
It is clear that ${\Ag}$ and ${\Ag^t}$ have the same number of elements and the same degree.

\begin{definition}{\rm
Let ${\Ag}\in \fg$. Then
\[ {\rm L}_{\Ag}=\{U:\dim U=1 \ \text{\rm and}\  \Ag\subseteq \Ag_U\}\]
is the set of all lines fixed by $\Ag$ and
\[ {\rm P}_{\Ag}=\{U:\dim U=2 \ \text{\rm and}\  \Ag\subseteq \Ag_U\}\]
is the set of all planes fixed by $\Ag$.}
\end{definition}
\noindent
Note that ${\rm L}_{\Ag^t}=\{\pi^\perp: \pi\in {\rm P}_{\Ag}\}$ and
${\rm P}_{\Ag^t}=\{l^\perp: l\in {\rm L}_{\Ag}\}$. Also, $X_{\Ag}={\rm L}_{\Ag}\cup {\rm P}_{\Ag}$.

\vspace{3mm}
\noindent
Consider an algebra ${\Ag}\in \fg$, $\Ag\neq \Dg$. Then $\Ag$ falls in exactly one of the following cases.

\vspace{3mm}
\noindent
{\bf Case I. ${\rm L}_{\Ag}$ contains three lines in general position.}
Recall that we say that three lines in $\Fg_q^3$ are in general position if they
are not contained in any plane. Dually, three planes are in general position if they
do not share any common line. Let $l_1,l_2,l_3\in {\rm L}_{\Ag}$ be three lines in
general position. Let $\pi_i$ be the plane spanned by $l_j$ and $l_k$, where
$\{i,j,k\}=\{1,2,3\}$.

\vspace{3mm}
\noindent
{\bf Subcase Ia. ${\rm L}_{\Ag}=\{l_1,l_2,l_3\}$.} In this case ${\rm P}_{\Ag}=\{\pi_1,\pi_2,\pi_3\}$.
The algebra $\Ag$ is conjugate to the algebra of all diagonal matrices. In particular, $|\Ag|=q^3$.
Furthermore, $X_{\Ag}={\rm L}_{\Ag}\cup {\rm P}_{\Ag}$ and a subset of $X_{\Ag}$ belongs to $\mathcal S_{\Ag}$
iff it contains one of the following sets: $\{l_1,l_2,l_3\}$, $\{\pi_1,\pi_2,\pi_3\}$, $\{l_1,l_2,\pi_1,\pi_2\}$,
$\{l_1,l_3,\pi_1,\pi_3\}$, $\{l_2,l_3,\pi_2,\pi_3\}$. Thus $\mathcal S_{\Ag}$ has 2 members of cardinality
$3$, nine members of cardinality $4$, six members of cardinality  $5$ and one element of cardinality $6$.
Therefore, $\d(\Ag)=-2+9-6+1=2$.

Note that the algebras in this subcase are in bijective correspondence with sets of three
lines in general position. Recall that $\Fg_q^3$ has $q^2+q+1$ lines, and each plane has $q+1$ lines.
It follows that the number of ordered triples of lines in general position is $(q^2+q+1)(q^2+q)q^2$.
Thus, the number of algebras in this subcase is $q^3(q+1)(q^2+q+1)/6$.
Consequently, the algebras in this subcase
contribute the quantity
%\[3^{-1}q^3(q+1)(q^2+q+1)(q^{3m}-q^{m})=
\[3^{-1}q^{m+3}(q+1)(q^2+q+1)(q^{2m}-1)\]
to the sum $\displaystyle \sum_{{\Ag}\in \fg}\d({\Ag})\left(|{\Ag}|^m
- q^m\right)$.

\vspace{3mm}
\noindent
{\bf Subcase Ib. ${\rm L}_{\Ag}\supseteq \{l_1,l_2,l_3,l_4\}$, where $l_4$ is a line not contained in any
of the planes $\pi_1,\pi_2,\pi_3$.} In this case, by Lemma \ref{scaling}, we have $\Ag=\Dg$, and $\Ag$
does not contribute anything to the sum $\displaystyle \sum_{{\Ag}\in \fg}\d({\Ag})\left(|{\Ag}|^m
- q^m\right)$.

It remains to consider the case when ${\rm L}_{\Ag}$ contains a line $l_4$ which is contained in one
of the planes $\pi_1,\pi_2,\pi_3$. Changing the numbering if necessary, we may assume that $l_4$ belongs
to $\pi_1$. If there is a line $l_5$ (different from $l_1,\ldots,l_4$) which is contained in $\pi_2$, the planes
through $l_4,l_1$ and through $l_5, l_2$ intersect along a line $l_6$ which does not belong to any
of the planes $\pi_1,\pi_2,\pi_3$. Thus we are in Subcase Ib. The same argument shows that there is
no line in ${\rm L}_{\Ag}$ different from $l_1,\ldots,l_4$ and contained in $\pi_3$. Since $\Ag$ fixes three
different lines in $\pi_1$, it acts as a scalar on $\pi_1$ by Lemma \ref{scaling}. In particular,
${\rm L}_{\Ag}$ contains all the lines in $\pi_1$. We will write $\pi$ for $\pi_1$ and $l$ for $l_1$.
We see that all the remaining algebras in Case I fall in the following subcase.

\vspace{3mm}
\noindent
{\bf Subcase Ic. ${\rm L}_{\Ag}=\{l\}\cup \{\text{\rm all lines in}\ \pi\}$.}
It is easy to see that in this case ${\rm P}_{\Ag}=\{\pi\}\cup \{\text{\rm all planes through}\ l\}$.
We will show that $\d(\Ag)=0$ by applying Lemma \ref{degree0} to $X=X_{\Ag}$, $\mathcal S=\mathcal S_{\Ag}$.
We need to verify that $x=l$, $y=\pi$ satisfy condition $(2)$. Suppose that $X'\in \mathcal S_{\Ag}$
and $X'-\{l\}\not\in \mathcal S_{\Ag}$. We claim that $X'$ contains at most one plane through $l$.
For suppose otherwise, that there are two planes containing $l$ in $X'$. Their
intersection is $l$. Thus a matrix fixes all elements of $X'-\{l\}$ if and only if
it fixes all elements of $X'$, i.e. $X'-\{l\}\in \mathcal S_{\Ag}$,  a contradiction. This proves
that indeed $X'$ contains at most one plane different from
$\pi$. We claim that $X'$ contains at lest two lines contained in $\pi$. Otherwise, there would be at most
one such line in $X'$, so $X'$ would be a subset of a set of the form $\{l,l',\pi,\pi'\}$ for some line
$l'$ contained in $\pi$
and some plane $\pi'$ containing $l$. Thus $\Ag$ would contain the algebra
$\Ag'=\Ag_l\cap \Ag_{l'}\cap\Ag_{\pi}\cap\Ag_{\pi'}$. This is however not possible, since $\Ag'$ has an element
which is not a scalar on $\pi$ and all elements of $\Ag$ act as scalars on $\pi$.
Indeed, if the line $l''=\pi\cap\pi'$ is different from $l'$ then $\Ag'=\Ag_l\cap \Ag_{l'}\cap\Ag_{l^{''}}$
and it contains the matrix which is identity on $l$ and $l'$ and is $0$ on $l''$. If $l''=l'$ then
$\Ag'=\Ag_l\cap \Ag_{l'}\cap\Ag_{\pi}$ contains the algebra $\Ag_l\cap \Ag_{l'}\cap\Ag_{l_1^{'}}$
for any line $l_1^{'}$ in $\pi$ which is different from $l'$.

Thus there are two lines in $X'$ which are contained in $\pi$.
These two lines span $\pi$, so $X'-\{\pi\}\in \mathcal S_{\Ag}$. Also, $(X'\cup \{\pi\})-\{l\}$
and $X'-\{l\}$ are fixed by the same set of matrices, so $(X'\cup\{\pi\})-\{l\}\not \in \mathcal S_{\Ag}$.
This verifies condition $(2)$ of Lemma \ref{degree0}, so $\d(\Ag)=0$. Consequently,
the algebras of Subcase Ic do not contribute anything to the sum
$\displaystyle \sum_{{\Ag}\in \fg}\d({\Ag})\left(|{\Ag}|^m- q^m\right)$.

\vspace{3mm}
Note that if ${\rm P}_{\Ag}$ contains three planes in general position, then the three lines obtained
by intersecting pairs of these planes are in general position and belong to ${\rm L}_{\Ag}$.
Thus from now on we assume that ${\rm L}_{\Ag}$ does not contain three lines
in general position and that ${\rm P}_{\Ag}$ does not contain three planes in general position.
If ${\rm L}_{\Ag}$ contains more than two elements, then all of the lines in ${\rm L}_{\Ag}$
must be contained in some plane $\pi$ and then, by Lemma \ref{scaling}, ${\rm L}_{\Ag}=\{\text{all lines in
$\pi$}\}$. Similarly, by duality, if ${\rm P}_{\Ag}$ contains more than two elements, then all the planes in
${\rm P}_{\Ag}$ share a common line $l$ and ${\rm P}_{\Ag}= \{\text{all planes which contain $l$}\}$.
This leads to the following two cases:

\vspace{3mm}
\noindent
{\bf Case II. There is a plane $\pi$ such that ${\rm L}_{\Ag}=\{\text{all lines in
$\pi$}\}$}. By Lemma \ref{scaling}, every element of $\Ag$ acts as a scalar on $\pi$. In particular,
$\pi\in {\rm P}_{\Ag}$. Note that all the planes in ${\rm P}_{\Ag}$ must share
a common line $l$ (if ${\rm P}_{\Ag}=\{\pi\}$, pick any line in $\pi$ for $l$).
In fact, suppose that there are $\pi_1,\pi_2 \in {\rm P}_{\Ag}$ such that the lines
$\pi\cap \pi_1$ and $\pi\cap \pi_2$ are different. Then the line $\pi_1\cap \pi_2$ belongs
to ${\rm L}_{\Ag}$ and it is not contained in $\pi$, which is not possible. Thus,
${\rm P}_{\Ag}\subseteq \{\text{all planes which contain $l$}\}$. We claim that any
$X\in \mathcal S_{\Ag}$ contains at lest two lines in $\pi$ different from $l$.
In fact, if the lines in $X$ are contained in $\{l,l_1\}$  then consider a
plane $\pi_1$ which does not contain $l$ but contains $\l_1$. There is a matrix
$A$ which is 0 on $l$ and is the identity on $\pi_1$ and this matrix fixes every plane passing through
$l$. Thus $A$ fixes all elements of $X$, yet $A$ is not a scalar on $\pi$. This means that $A\not\in \Ag$,
and consequently $X\not\in S_{\Ag}$, a contradiction. Now any two lines in $X$ span $\pi$. It follows
that any matrix which fixes all elements of $X-\{\pi\}$ also fixes $\pi$, i.e. $X-\{\pi\}\in S_{\Ag}$.
This means that he family $\mathcal S_{\Ag}$ of subsets of $X_{\Ag}$ satisfies the assumptions of Lemma \ref{degree0},
case $(1)$, with $x=\pi$. It follows that $\d(\Ag)=0$ and the algebras in this case do not contribute anything to
$\displaystyle \sum_{{\Ag}\in \fg}\d({\Ag})\left(|{\Ag}|^m- q^m\right)$.

\vspace{3mm}
\noindent
{\bf Case III. There is a lane $l$ such that ${\rm P}_{\Ag}=\{\text{all planes through $l$}\}$.}
Any algebra in this case is dual to an algebra in Case II, hence it has degree 0. Thus
algebras in this case do not contribute anything to
$\displaystyle \sum_{{\Ag}\in \fg}\d({\Ag})\left(|{\Ag}|^m- q^m\right)$.

\vspace{3mm}
\noindent
It remains to analyze algebras $\Ag$ such that both ${\rm L}_{\Ag}$ and ${\rm P}_{\Ag}$
have at most two elements.

\vspace{3mm}
\noindent
{\bf Case IV. $|{\rm L}_{\Ag}|=2=|{\rm P}_{\Ag}|$.} We may assume that ${\rm L}_{\Ag}=\{l,l'\}$
and ${\rm P}_{\Ag}=\{\pi,\pi'\}$, where $\pi'$ is spanned by $l,l'$ and $\pi\cap \pi'=l$.
It is easy to see that the family $\mathcal S_{\Ag}$ has three elements: $\{l,l',\pi\}$, $\{l',\pi,\pi'\}$,
and $\{l,l',\pi,\pi'\}$. Thus $\d(\Ag)=-2+1=-1$. Choosing non-zero vectors $v_1\in l'$, $v_2\in l$, and
$v_3\in \pi-l$ we get a basis of $\Fg_q^3$ and $A\in \Ag$ if and only if the matrix of the linear
transformation given by $A$, expressed in the basis $v_1,v_2,v_3$, has the form $\displaystyle \left(
\begin{smallmatrix}
* & 0 & 0\\ 0 & * & *\\ 0 & 0& *
\end{smallmatrix} \right) $.
In other words, $\Ag$ is conjugate to the algebra of all the matrices of the form
$\displaystyle \left(\begin{smallmatrix}
* & 0 & 0\\ 0 & * & *\\ 0 & 0& *
\end{smallmatrix} \right) $. In particular, $|\Ag|=q^4$. To count the number of algebras in Case IV,
note that these algebras are in bijective correspondence with triples $l,l',\pi$, where $\pi$ is
a plane and $l,l'$ are lines such that $l\subset \pi$ and $l'\not\subset \pi$. There are $q^2+q+1$
choices for $\pi$ and for each $\pi$ we have $q+1$ choices of $l$ and $q^2$ choices of $l'$. Thus the number
of algebras in Case IV is $q^2(q+1)(q^2+q+1)$. Consequently, the algebras in this case contribute
%\[ -q^2(q+1)(q^2+q+1)(q^{4m}-q^m)=
\[-q^{m+2}(q+1)(q^2+q+1)(q^{3m}-1)\]
to $\displaystyle \sum_{{\Ag}\in \fg}\d({\Ag})\left(|{\Ag}|^m- q^m\right)$.

\vspace{3mm}
\noindent
{\bf Case V. $|{\rm L}_{\Ag}|=2$ and $|{\rm P}_{\Ag}|=1$.} Thus ${\rm L}_{\Ag}=\{l,l'\}$
and ${\rm P}_{\Ag}=\{\pi\}$, where $\pi$ is spanned by $l,l'$. It is straightforward to see that
$\mathcal S_{\Ag}$ has two elements: $\{l,l'\}$ and  $\{l,l',\pi\}$. It follows that $\d(\Ag)=0$
and therefore algebras in this case contribute nothing to
$\displaystyle \sum_{{\Ag}\in \fg}\d({\Ag})\left(|{\Ag}|^m- q^m\right)$.

\vspace{3mm}
\noindent
{\bf Case V$^\perp$. $|{\rm L}_{\Ag}|=1$ and $|{\rm P}_{\Ag}|=2$.}
Algebras in this case are dual to algebras in Case V, so they have degree 0 and
contribute nothing to
$\displaystyle \sum_{{\Ag}\in \fg}\d({\Ag})\left(|{\Ag}|^m- q^m\right)$.

\vspace{3mm}
\noindent
{\bf Case VI. ${\rm L}_{\Ag}=\{l\}$ and ${\rm P}_{\Ag}=\{\pi\}$, where $l\not\subset \pi$.}
It is clear that $\mathcal S_{\Ag}$ has exactly one element: $\{l,\pi\}$. Thus $\d(\Ag)=1$.
Choosing a basis $v_1$ of $l$ and $v_2,v_3$ of $\pi$ we easily see that $\Ag$ is conjugate
to the algebra of all the matrices of the form
$\displaystyle \left(\begin{smallmatrix}
* & 0 & 0\\ 0 & * & *\\ 0 & *& *
\end{smallmatrix} \right) $. In particular, $|\Ag|=q^5$. To count the number of algebras in Case VI,
note that these algebras are in bijective correspondence with pairs $l,\pi$, where $\pi$ is
a plane and $l$ is a line not contained in $ \pi$. There are $q^2+q+1$
choices for $\pi$ and for each $\pi$ we have $q^2$ choices of $l$. Thus the number
of algebras in Case VI is $q^2(q^2+q+1)$. Consequently, the algebras in this case contribute
%\[q^2(q^2+q+1)(q^{5m}-q^m)=
\[q^{m+2}(q^2+q+1)(q^{4m}-1)\]
to
$\displaystyle \sum_{{\Ag}\in \fg}\d({\Ag})\left(|{\Ag}|^m- q^m\right)$.

\vspace{3mm}
\noindent
{\bf Case VII. ${\rm L}_{\Ag}=\{l\}$ and ${\rm P}_{\Ag}=\{\pi\}$, where $l\subset \pi$.}
It is clear that $\mathcal S_{\Ag}$ has exactly one element: $\{l,\pi\}$. Thus $\d(\Ag)=1$.
Choosing a basis $v_1$ of $l$, $v_1,v_2$ of $\pi$, and a vector $v_3\not\in\pi$, we easily see that
$\Ag$ is conjugate
to the algebra of all the matrices of the form
$\displaystyle \left(\begin{smallmatrix}
* & * & *\\ 0 & * & *\\ 0 & 0 & *
\end{smallmatrix} \right) $. In particular, $|\Ag|=q^6$. To count the number of algebras in Case VII,
note that these algebras are in bijective correspondence with pairs $l,\pi$, where $\pi$ is
a plane and $l$ is a line contained in $ \pi$. There are $q^2+q+1$
choices for $\pi$ and for each $\pi$ we have $q+1$ choices of $l$. Thus the number
of algebras in Case VII is $(q+1)(q^2+q+1)$. Consequently, the algebras in this case contribute
%\[(q+1)(q^2+q+1)(q^{6m}-q^m)=
\[q^{m}(q+1)(q^2+q+1)(q^{5m}-1)\]
to $\displaystyle \sum_{{\Ag}\in \fg}\d({\Ag})\left(|{\Ag}|^m- q^m\right)$.

\vspace{3mm}
\noindent
{\bf Case VIII. ${\Ag}={\Ag}_{l}$ for some line $l$.} The family
$\mathcal S_{\Ag}$ has exactly one element: $\{l \}$, so $\d(\Ag)=-1$.
It is easy to see that $\Ag$ is conjugate to the algebra of all the matrices of the form
$\displaystyle \left(\begin{smallmatrix}
* & * & *\\ 0 & * & *\\ 0 & * & *
\end{smallmatrix} \right) $. In particular, $|\Ag|=q^7$. Algebras in this case are in bijection with
lines, so we have $q^2+q+1$ such algebras. Thus the algebras in this case contribute
%\[ -(q^2+q+1)(q^{7m}-q^m)=
\[-q^m(q^2+q+1)(q^{6m}-1)\]
to $\displaystyle \sum_{{\Ag}\in \fg}\d({\Ag})\left(|{\Ag}|^m- q^m\right)$.

\vspace{3mm}
\noindent
The following case is the last to consider.

\vspace{3mm}
\noindent
{\bf Case VIII$^\perp$. ${\Ag}={\Ag}_{\pi}$ for some plane $\pi$.} This case consists of algebras
dual to algebras of Case VIII, so they also contribute
\[-q^m(q^2+q+1)(q^{6m}-1)\]
to $\displaystyle \sum_{{\Ag}\in \fg}\d({\Ag})\left(|{\Ag}|^m- q^m\right)$.

\vspace{3mm}
\noindent
Putting together all the contributions to $\displaystyle \sum_{{\Ag}\in \fg}\d({\Ag})\left(|{\Ag}|^m- q^m\right)$
we arrive at the following formula:
\begin{eqnarray} \nonumber
%\g_{m,3}(q)=  q^{9m}-q^m-3^{-1}(q^{3m}-q^m)(q^3-q)(q^3-q^2)+\\ \nonumber
\sum_{{\Ag}\in \fg}\d({\Ag})\left(|{\Ag}|^m- q^m\right)=
3^{-1}q^{m+3}(q+1)(q^2+q+1)(q^{2m}-1)-\\ \nonumber  q^{m+2}(q+1)(q^2+q+1)(q^{3m}-1)+
q^{m+2}(q^2+q+1)(q^{4m}-1)+\\ \nonumber  q^{m}(q+1)(q^2+q+1)(q^{5m}-1)-2q^m(q^2+q+1)(q^{6m}-1).
\end{eqnarray}
%which is the same as
%\begin{eqnarray} \nonumber
%\g_{m,3}(q)=  q^{9m}-q^m+q^{m+4}(q+1)(q^{2m}-1)-\\ \nonumber  q^{m+2}(q+1)(q^2+q+1)(q^{3m}-1)+
%q^{m+2}(q^2+q+1)(q^{4m}-1)+\\ \nonumber  q^{m}(q+1)(q^2+q+1)(q^{5m}-1)-2q^m(q^2+q+1)(q^{6m}-1)
%\end{eqnarray}
%which is
%\begin{eqnarray} \nonumber
%\g_{m,3}(q)=  q^{9m}+q^{m+4}(q+1)q^{2m}-\\ \nonumber  q^{m+2}(q+1)(q^2+q+1)q^{3m}+
%q^{m+2}(q^2+q+1)q^{4m}+\\ \nonumber  q^{m}(q+1)(q^2+q+1)q^{5m}-2q^m(q^2+q+1)q^{6m}
%\end{eqnarray}
After inserting this into (\ref{preparation1}) and simplifying we arrive at the following formula for $\g_{m,3}(q)$:
\begin{eqnarray} \label{n=3}
\g_{m,3}(q)= q^{3m+4}(q^{m-1}-1)(q^{m-1}+1)(q^m-1)\times \\ \nonumber (q^{3m-2}+q^{2m-2}-q^{m}-2q^{m-1}-q^{m-2}+q+1).
\end{eqnarray}

\subsection{Lower bound for $\g_{m,n}(q)$}
So far we have been unable to obtain exact formulas for $\g_{m,n}(q)$ for any $n\geq 4$. We have
however the following lower bound.

\begin{pro}\label{lowerbound} Let $m,n$ be positive integers and let $q$ be a power of a prime number. Then
\begin{equation}
\g_{m,n}(q)\geq q^{mn^2}-2^{\frac{n+6}{2}}q^{n^2m-(m-1)(n-1)}.
\end{equation}
\end{pro}
\begin{proof}
By (\ref{avoidmax}), we have the following inequality
\[\g_{m,n}(q)\geq q^{mn^2}-\sum |\Ag|^m,\]
where the sum is taken over all maximal subalgebras $\Ag$ of $\mt{n}$. We use the description of maximal
subalgebras given by Proposition \ref{clasmax}.
Let $1\leq k<n$. The number of $k$-dimensional subspaces of $\Fg_q^n$ is
$\prod_{i=0}^{k-1}(q^n-q^{i})\prod_{i=0}^{k-1}(q^k-q^{i})^{-1}$. For any such subspace $V$ the algebra
$\Ag_V$ has $q^{n^2-nk+k^2}$ elements. Let $S_k=\sum |\Ag_V|^m$, where the sum
is taken over all $k$-dimensional subspaces $V$ of $\Fg_q^n$. It follows that
\[S_k =q^{(n^2-nk+k^2)m}\prod_{i=0}^{k-1}(q^n-q^{i})\prod_{i=0}^{k-1}(q^k-q^{i})^{-1}.
\]
Using the inequality  $\displaystyle \frac{q^n-q^i}{q^k-q^i}\leq q^{n-k}\frac{q}{q-1}$ we get
\[S_k\leq q^{n^2m-(m-1)k(n-k)}\left(\frac{q}{q-1}\right)^k.\]
Note that $S_k=S_{n-k}$ (by duality). Since $\frac{q}{q-1}\leq 2$ and $k(n-k)\geq n-1$, we have
\[\Sigma_a:=\sum_{k=1}^{n-1}S_k\leq 2 \sum_{k=1}^{\lfloor \frac{n}{2}\rfloor}S_k\leq
2^{\frac{n+4}{2}}q^{n^2m-(m-1)(n-1)}.\]
For a prime divisor $s$ of $n$ define $T_s$ as the sum $\sum|\Ag|^m$, where the sum is over all
subalgebras of $\mt{n}$ isomorphic to $\mx{n/s}{s}$. By \ref{fieldmax}, we have
\[T_s=q^{\frac{n^2}{s}m}\cdot s^{-1}\cdot \prod_{s\nmid i, 1\leq i<n}(q^n-q^i)\leq s^{-1}\cdot q^{\frac{n^2}{s}m}\cdot
q^{n(n-\frac{n}{s})}\leq s^{-1}\cdot q^{n^2\frac{m+1}{2}}.\]
Let $\Sigma_b=\sum T_s$, where the sum is over all prime divisors $s$ of $n$.
It is easy to see that the sum $\sum s^{-1}$ of all reciprocals of prime divisors of $n$
does not exceed $2^{(n+4)/2}$. Furthermore, $q^{n^2(m+1)/2}\leq q^{n^2m-(m-1)(n-1)}$.
It follows that
\[ \Sigma_b\leq 2^{\frac{n+4}{2}}q^{n^2m-(m-1)(n-1)}.\]
By Proposition \ref{clasmax} we have
$\Sigma_a+\Sigma_b=\sum |\Ag|^m$,
where the sum is taken over all maximal subalgebras $\Ag$ of $\mt{n}$. Thus,
\[ \g_{m,n}(q)\geq q^{mn^2}-2^{\frac{n+6}{2}}q^{n^2m-(m-1)(n-1)}.\]
\end{proof}
As an immediate consequence of Proposition \ref{lowerbound} we get the following corollary.

\begin{corollary}\label{infinity}
Let $m,n \geq 2$. The probability that $m$ matrices
in $\M_n \left( \mathbb{F}_q \right)$, chosen under the uniform
distribution, generate the $\mathbb{F}_q$-algebra $\M_n \left( \mathbb{F}_q \right)$
tends to $1$ as $q+m+n \to \infty$.
\end{corollary}

\noindent
Corollary~\ref{infinity} proves and vastly generalizes the conjectural formula (17)
on. page~27 of \cite{pet-sid}.

\section{Finite products of matrix algebras over rings of algebraic integers}\label{numfield}
Let $R$ be the ring of integers in a number field $K$. In this final section we apply the techniques developed
in our paper to investigate generators of $R$-algebras $A$ which are
products of a finite number of matrix algebras over $R$. Thus we have
\[ A\cong \prod_{i=1}^{s} \M_{n_i}(R)^{m_i},\]
where $1\leq n_1<n_2<\ldots<n_s$ and $m_i$ are positive integers.
As we have seen in Example \ref{matrixproduct},
the algebra $A$ is $k$-generated if and only if all the algebras $\M_{n_i}(R)^{m_i}$ are
$k$-generated. Thus, we may and will focus on the case when $A\cong \M_{n}(R)^{m}$ for some
positive integers $n,m$. We have the following theorem.

\begin{theorem}\label{minnumbergen} Let $R$ be the ring of integers in a number field $K$.
Suppose that either $n\geq 3$ or $k\geq 3$ and let $A=\M_{n}(R)^{m}$ for some positive integer $m$.
Then the following conditions are equivalent.
\begin{enumerate}[\rm (i)]
\item The $R$-algebra $A$ admits $k$ generators.

\item For every maximal ideal $\f p$ of $R$ the $R/\f p$-algebra $\M_{n}(R/\f p)^{m}$ admits $k$ generators.

\item The density $\den_k(A)$ is positive.
\end{enumerate}
Furthermore, the following formulas, in which $\zeta_K$ denotes the Dedekind zeta-function of $K$, hold for
every $k\geq 2$.
\begin{enumerate}[\rm (a)]
\item $\den_2(\M_2(R)^m)=0$ for every $m$.

\item $\displaystyle \den_k(\M_2(R))=\frac{1}{\zeta_K(k-1)\zeta_K(k)}$.

\item $\displaystyle \den_k(\M_3(R))=\frac{1}{\zeta_K(2k-2)\zeta_K(k)}\prod_{\f p\in \mspec R}\left(1+
\frac{\phi_k(\N(\f p))}{\N(\f p)^{3k-2}}\right)$, where $\phi_k(x)=x^{2k-2}-x^{k}-2x^{k-1}-x^{k-2}+x+1$.
\end{enumerate}
\end{theorem}

\begin{proof}
The implications $\text{(i)}\Rightarrow \text{(ii)}$ and $\text{(iii)}\Rightarrow \text{(i)}$
are clear. When $k\geq 3$, the implication $\text{(ii)}\Rightarrow \text{(iii)}$
is an immediate consequence of Theorem~\ref{denslenstra} and the fact that
the $K$-algebra $\M_{n}(K)^{m}$ is 2-generated (\cite{mmbp}).
Suppose now that $k=2$, $n\geq 3$ and (ii) holds.
Consider a maximal ideal $\f p$ of $R$ and let $q=\N(\f p)$. By Theorem \ref{formula}, the number
$\g_2(\f p, A)$ of pairs of elements which generate $\M_{n}(R/\f p)^{m}$ is given by
\[ \g_2(\f p, A)=\prod_{i=0}^{m-1}(\g_{2,n}(q)-i\cdot |\text{\rm PGL}_n(\mathbb F_{q})|).\]
By (ii), we have $g_2(\f p, A)>0$.
Note that $|\text{\rm PGL}_n(\mathbb F_{q})|\leq q^{n^2-1}\leq q^{2n^2-n+1}$. Furthermore,
we have $\g_{2,n}(q)\geq q^{2n^2}-2^{n}q^{2n^2-n+1}$ by Proposition \ref{lowerbound}.
It follows that
\[\g_2(\f p, A)\geq (\N(\f p)^{2n^2}-(2^n+m)\N(\f p)^{2n^2-n+1})^m \]
provided $\N(\f p)>2^n+m$.
By Theorem~\ref{gendensity}, we have
\[ \den_2(A)=\prod_{\f p\in \mspec R}\frac{\g_2(\f p, A)}{\N(\f p)^{2mn^2}}.\]
Since all the factors in the product on the right are positive and all but a finite number of them
satisfy the inequality
\[\frac{\g_2(\f p, A)}{\N(\f p)^{2mn^2}}\geq \left(1-\frac{m+2^n}{\N(\f p)^{n-1}}  \right)^m ,\]
the product converges to a positive number. In other words, $\den_2(A)>0$. This completes
the proof of the implication $\text{(ii)}\Rightarrow \text{(iii)}$.

In order to establish the formulas (b) and (c) note that
\[ \den_k(\M_n(R))=\prod_{\f p\in \mspec R}\frac{\g_{k,n}(\N(\f p))}{\N(\f p)^{kn^2}}\]
by Theorem~\ref{gendensity}. The formulas (b) and (c) follow now from (\ref{n=2}) and (\ref{n=3})
respectively. To justify (a) note that $\g_2(\f p, \M_2(R)^m)\leq \g_{2,2}(\N(\f p))^m$ for every maximal
ideal $\f p$. It follows that $\den_2(\M_2(R)^m)\leq \den_2(\M_2(R))^m$. Since by (b) with $k=2$
we have $\den_2(\M_2(R))=0$, the equality in (a) follows. \end{proof}

Recall now that by Theorem \ref{formula}, the $R/\f p$-algebra $\M_n(R/\f p)^m$ is $k$-generated
iff $m\leq \gen_{k,n}(\N(\f q))$, where $\displaystyle \gen_{k,n}(q)=\frac{\g_{k,n}(q)}{|\text{\rm PGL}_n(\mathbb F_{q})| }$.
Using the formula (\ref{n=3}) we get the following theorem.

\begin{theorem}\label{3by3} Let $R$ be the ring of integers in a number field and let $\f p$ be a maximal
ideal of $R$ with smallest norm. Define polynomials $f_k(x)$ by $f_1(x)=0$ and
\begin{eqnarray}
f_k(x)= \frac{x^{3k+1}(x^{k-1}-1)(x^{k-1}+1)(x^k-1)}
{(x^2+x+1)(x-1)^2(x+1)}\times\\ \nonumber (x^{3k-2}+x^{2k-2}-x^{k}-2x^{k-1}-x^{k-2}+x+1)
\end{eqnarray}
for any $k\geq 2$.
Let $k\geq 2$ and $m$ be positive integers. Then the following conditions are equivalent.
\begin{enumerate}[\rm (i)]
\item $r(\M_3(R)^m, R)=k$.

\item $f_{k-1}(\N(\f p)) < m \leq f_k(\N(\f p))$.
\end{enumerate}
In particular, the $\mathbb Z$-algebra $\M_3(\mathbb Z)^m$ is $2$-generated if and only if  $m\leq 768$.
\end{theorem}
\begin{proof}
By (\ref{n=3}), we have $\gen_{k,3}(q)=f_k(q)$ for any $k\geq 2$.
By Theorem \ref{minnumbergen}, the $R$-algebra
$\M_3(R)^m$ is $k$-generated iff $m\leq f_k(\N(\f q))$ for every maximal ideal $\f q$ of $R$.
It is easy
to see that $f_k(x)$ is increasing on $[2,\infty)$. It follows that
$\M_3(R)^m$ is $k$-generated iff $m\leq f_k(\N(\f p))$. This establishes the
equivalence of (i) and (ii). The last claim follows now from the fact that $f_2(2)=768$.
\end{proof}

Even though in  (\ref{n=2}) we established a formula for $\gen_{k,2}(q)$, getting
an analog of Theorem~\ref{3by3} for products of copies of $\M_2(R)$ is more complicated.
The difficulty is that the density $\den_2(\M_2(R)^m)$ is 0 and we have to find
a way to deal with the ambiguity in Theorem \ref{lenstra} when $k=2$. So far we can overcome this
difficulty only when $R$ has a maximal ideal of norm 2. We have the following theorem.

\begin{theorem}\label{23}
Let $R$ be the ring of integers in a number field and let $\f p$ be a maximal
ideal of $R$ with smallest norm. Define polynomials $h_k(x)$ by $h_1(x)=0$ and
\begin{eqnarray}
h_k(x)= \frac{x^{2k}(x^{k-1}-1)(x^k-1)}
{(x-1)(x+1)}
\end{eqnarray}
for any $k\geq 2$. Let $k>3$ and $m$ be positive integers. Then the following
conditions are equivalent.
\begin{enumerate}[\rm (i)]
\item $r(\M_2(R)^m, R)=k$.

\item $h_{k-1}(\N(\f p)) < m \leq h_k(\N(\f p))$.
\end{enumerate}
Furthermore, there is an integer $t$ such that $16 \leq t \leq h_2(\N(\f p))$, $\M_2(R)^m$ is $2$-generated
if and only if  $m\leq t$, and $r(\M_2(R)^m, R)=3$ if and only if $t<m\leq h_3(\N(\f p))$. In particular,
if $\N(\f p)=2$, then $t=16$ so in this case {\rm (i)} and {\rm (ii)} are equivalent for all
$k\geq 2$.
\end{theorem}
\begin{proof}
By (\ref{n=2}), we have $\gen_{k,2}(q)=h_k(q)$ for any $k\geq 2$. Suppose that $k\geq 3$.
By Theorem \ref{minnumbergen}, the $R$-algebra
$\M_2(R)^m$ is $k$-generated if and only if $m\leq h_k(\N(\f q))$ for every maximal ideal $\f q$ of $R$.
It is easy
to see that $h_k(x)$ is increasing on $[2,\infty)$. It follows that when $k\geq 3$ then
$\M_2(R)^m$ is $k$-generated if and only if $m\leq h_k(\N(\f p))$. This, in particular, justifies
the equivalence of (i) and (ii) when $k>3$. It also implies the existence of $t$
having all the required properties except possibly the estimate $t\geq 16$.
In order to show that $t\geq 16$, we need to establish that $\M_2(R)^{16}$ is $2$-generated
as an $R$-algebra. It suffices to prove that $\M_2(\mathbb Z)^{16}$ admits
two generators as an $\mathbb Z$-algebra. This will be done in Proposition \ref{2by2}.
Finally, the equality $t=16$ when $\N(\f p)=2$ follows from the fact that $h_2(2)=16$.
\end{proof}

In order to improve on Theorem~\ref{23} and extend it to matrix algebras of size $n\geq 3$ the
following two questions need to be answered.

\begin{question}{\rm
Is it true that $t=h_2(\N(\f p))$?}
\end{question}

\begin{question}
{\rm Given positive integers $k$ and $n$, is $\gen_{k,n}(q)$ an increasing function of $q$?}
\end{question}

It order to complete our proof of Theorem~\ref{23} we have to
show that $\M_2(\mathbb Z)^{16}$ admits two generators.
For that we need several
observations, which seem of independent interest.

\begin{proposition}\label{det}
Let $S$ be a commutative ring. Two matrices $A,B\in \M_2(S)$ generate $\M_2(S)$ as an $S$-algebra
if and only if $\det(AB-BA)$ is invertible in $S$.
\end{proposition}
\begin{proof}
First we prove the result under the additional assumption that $S$ is a field.
Let $N=AB-BA$ and let $T$ be the subalgebra generated by $A$ and $B$.
If $T$ is a proper subalgebra then its dimension is at most 3. It follows that $T/J(T)$ is a semi-simple
algebra of dimension $\leq 3$ (recall that $J(T)$ denotes the Jacobson radical of $T$).
Thus $T/J(T)$ is abelian and therefore $N\in J(T)$. Since the Jacobson radical
is nilpotent, $N$ is nilpotent, hence $\det N=0$.

Conversely, suppose that $\det N=0$. Recall that any $2$ by $2$ matrix $X$ satisfies the identity
$X^2=t_XX-d_XI$, where $t_X$ is the trace and $d_X$ is the determinant of $X$. Since the trace of $N$ is $0$,
we have $N^2=0$. If $N=0$ then $T$ is commutative, hence a proper subalgebra of
$\M_2(S)$. If $N\neq 0$, , the null-space of $N$ is one dimensional.
Using the identity $A^2=t_AA-d_AI$ we easily see that $AN+NA=t_AN$. It follows that
the null-space of $N$ is $A$-invariant. Similarly, the null-space of $N$ is $B$-invariant.
It follows that the null-space of $N$ is $T$-invariant, hence $T$
is a proper subalgebra of $\M_2(S)$. This completes our proof in the case when $S$ is a field.

If $S$ is any commutative ring then, by Lemma \ref{localgen}, the matrices $A, B$ generate $\M_2(S)$
if and only if for any maximal ideal $M$ of $S$ the $S/M$-algebra $\M_2(S/M)$ is generated by the
images of $A$ and $B$. By the just established field case of the result, this is equivalent to
the condition that $\det(AB-BA)\not\in M$ for all maximal ideals $M$, which in turn is
equivalent to claiming that $\det(AB-BA)$ is invertible in $S$.
\end{proof}

The following observation is due to H.W. Lenstra.

\begin{lemma}\label{lift}
Let $A,B\in \M_2(\mathbb Z)$ be two matrices with all entries in $\{0,1\}$. Then $A,B$ generate
$\M_2(\mathbb Z)$ if and only if their reductions modulo $2$ generate $\M_2(\mathbb F_2)$.
\end{lemma}
\begin{proof}
By Proposition \ref{det}, we need to prove that $\det(AB-BA)$ is odd if and only iff
it is $\pm 1$. The ``if'' part is clear.
Suppose then that $A=\displaystyle \left(\begin{smallmatrix}
a_1& a_2\\ a_3& a_4
\end{smallmatrix} \right) $ and $B=\displaystyle \left(\begin{smallmatrix}
b_1& b_2\\ b_3& b_4
\end{smallmatrix} \right) $  are such that $a_i, b_j\in \{0,1\}$ and
$\det(AB-BA)$ is odd. Note that $AB-BA=\displaystyle \left(\begin{smallmatrix}
a_2b_3-a_3b_2& a_2(b_4-b_1)+b_2(a_1-a_4)\\ a_3(b_1-b_4)+b_3(a_4-a_1)& a_3b_2-a_2b_3
\end{smallmatrix} \right) $. The diagonal entries of this matrix are in $\{0,\pm 1\}$ and the
off-diagonal entries are in the set $\{0,\pm 1,\pm 2\}$. If $a_2b_3-a_3b_2=0$ then the off-diagonal
entries must be odd and hence are $\pm 1$. It follows that $\det(AB-BA)=\pm 1$. The same conclusion holds if
one of the off-diagonal entries is 0.
Suppose now that $a_2b_3-a_3b_2=\pm 1$ and the off-diagonal entries are not 0.
Then one of the off-diagonal entries, say $a_3(b_1-b_4)+b_3(a_4-a_1)$,  must be even and non-zero
(the other possibility is handled in the same way).
This can only happen if $a_3=b_3=1$ and $b_1-b_4=a_4-a_1=\pm 1$. It follows that one of $a_2,b_2$
is 0 and the other is $1$. Thus $\det(AB-BA)=-1-(\mp 1)(\pm 2)=1$.
\end{proof}

\begin{lemma}\label{modp}
Let $A, B, A',B'\in \M_2(\mathbb Z)$ be matrices with all entries in $\{0,1\}$ such that each pair $A,B$ and
$A', B'$ generates $\M_2(\mathbb Z)$. If there is an odd prime $p$ such that the reductions modulo $p$ of $(A,B)$
and $(A', B')$ are conjugate in $\M_2(\mathbb F_p)$ then the pairs $(A,B)$ and $(A',B')$ are conjugate
in $\M_2(\mathbb Z)$.
\end{lemma}

\begin{proof}
For a pair of $2\times 2$ matrices $X,Y$  define
\[\conj(X,Y) = \left(\tr(X),\, \det(X), \,\tr(Y), \, \det(Y), \, \tr(XY) \right).\]
It follows from \cite[Theorem 2]{mmbp} that for any principal ideal domain $R$ and any two pairs
$(X,Y)$, $(X', Y')$ of elements in $\M_2(R)$ which generate $\M_2(R)$ as an $R$-algebra
we have $\conj(X,Y)=\conj(X', Y')$ if and only if $X'=CXC^{-1}$ and $Y'=CYC^{-1}$
for some invertible matrix $C\in \M_2(R)$ (in \cite{mmbp} the fifth component of $\conj$
is $\det(X+Y)$ but it is equivalent to the version above by the following identity for
$2\times 2$ matrices
\[ \tr(X) \tr(Y)-\tr(XY) +\det(X)+\det(Y)-\det(X+Y) = 0 .)
\]
Under the assumptions of the lemma, the traces of $A,B,A',B'$ are in $\{0,1,2\}$ and
the determinants of these matrices are in $\{-1,0,1\}$. Our assumption that $\conj(A,B)\equiv\conj(A',B')\mo p$
implies then that $\tr A=\tr A'$, $\det A=\det A'$, $\tr B=\tr B'$ and $\det B=\det B'$. It remains to prove that
$\tr(AB)=\tr(A'B')$. Let $A=\displaystyle \left(\begin{smallmatrix}
a_1& a_2\\ a_3& a_4
\end{smallmatrix} \right) $, $B=\displaystyle \left(\begin{smallmatrix}
b_1& b_2\\ b_3& b_4
\end{smallmatrix} \right) $,$A'=\displaystyle \left(\begin{smallmatrix}
a_1^{'}& a_2^{'}\\ a_3^{'}& a_4^{'}
\end{smallmatrix} \right) $, and $B'=\displaystyle \left(\begin{smallmatrix}
b_1^{'}& b_2^{'}\\ b_3^{'}& b_4^{'}
\end{smallmatrix} \right) $. Then $\tr(AB)=a_1b_1+a_2b_3+a_3b_2+a_4b_4$ and
$\tr(A'B')=a_1^{'}b_1^{'}+a_2^{'}b_3^{'}+a_3^{'}b_2^{'}+a_4^{'}b_4^{'}$. Both these
numbers belong to $\{0,1,2,3,4\}$. Suppose that these numbers are different. Since they are congruent
modulo $p$, we see that
$p=3$ and one of these numbers is in $\{0,4\}$. If $\tr(AB)=4$ then all the entries $a_i$ and $b_j$
must be $1$ so $A=B$, which is not possible. Thus we may assume that $\tr(AB)=0$ and then $\tr(A'B')=3$.
If $\tr(A)=0$ then $\tr(A')=0$, so $a_1^{'}=a_4^{'}=0$ and therefore $\tr(A'B')\leq 2$, a contradiction.
Thus $\tr(A)\neq 0$ and in the same way we show that $\tr(B)\neq 0$. If $\tr(A)=2$ then $a_1=a_4=1$ so $b_1=b_4=0$ and
$\tr(B)=0$, which we have just proved impossible. This shows that $\tr(A)=1$ and similar argument
yields $\tr(B)=1$. Thus
$\tr(A')=1=\tr(B')$. It follows that one of $a_1^{'}$, $a_4^{'}$ is 0. We may assume that $a_1^{'}=0$
(same argument works when $a_4^{'}=0$).
Then $a_2^{'}=a_3^{'}=a_4^{'}=b_2^{'}=b_3^{'}=b_4^{'}=1$ and consequently $b_1^{'}=0$ and $A'=B'$,
 a contradiction.
\end{proof}

We have now the following curious proposition.

\begin{proposition}\label{2by2}
Let $x,y$ be two elements of $\M_2(\mathbb Z)^{k}$ such that every component of $x$ and $y$ is a matrix
whose all entries are in $\{0,1\}$. Suppose that $x,y$, considered as elements of
$\M_2(\mathbb F_2)^{k}$, generate the algebra $\M_2(\mathbb F_2)^{k}$.
Then $x,y$ generate $\M_2(\mathbb Z)^{k}$ as a ring.
In particular,
the ring $\M_2(\mathbb Z)^{16}$ admits $2$ generators.
\end{proposition}
\begin{proof}
Let $x=(X_1,\ldots,X_k)$, $y=(Y_1,\ldots,Y_k)$. By Lemma \ref{lift}, each pair $(X_i,Y_i)$ generates
$\M_2(\mathbb Z)$. According to Lemma~\ref{localgen} and Theorem~\ref{simple}, it suffices
to prove that for any prime $p$ and any $1\leq i<j\leq k$, the pairs $(X_i,Y_i)$ and $(X_j, Y_j)$
are not conjugate modulo $p$. For $p=2$ this follows from our assumptions and Theorem~\ref{simple}.
Consequently, the pairs $(X_i,Y_i)$ and $(X_j, Y_j)$
are not conjugate in $\M_2(\mathbb Z)$ whenever $i\neq j$.
By Lemma~\ref{modp}, the pairs $(X_i,Y_i)$ and $(X_j, Y_j)$
are not conjugate modulo $p$ for any odd prime $p$. This proves the first part
of the proposition.

Since $\gen_{2,2}(2)=16$ by (\ref{n=2}), the algebra
$\M_2(\mathbb F_2)^{16}$ is two generated. It follows from the first part of the proposition that
$\M_2(\mathbb Z)^{16}$ admits $2$ generators.
\end{proof}

\noindent
\begin{remark} {\rm We would like to point out that one should not expect any analogs of Proposition~\ref{2by2}
for matrix rings of size larger than $2$. For example, consider the matrices
$ \displaystyle A=\left(\begin{smallmatrix}
0&0& 0\\ 0& 0&0\\0&1&1
\end{smallmatrix} \right) \ \ \text{and}\ \ B=\left(\begin{smallmatrix}
0&0& 1\\ 1& 0&1\\0&0&1
\end{smallmatrix} \right)$. Considered as matrices over the field $\mathbb F_3$ with $3$ elements
they have a common eigenvector $(1,-1,1)^t$. Thus these matrices do not generate $\M_3(\mathbb F_3)$,
hence they do not generate $\M_3(\mathbb Z)$. Consider now these matrices as matrices over $\mathbb F_2$.
If they do not generate $\M_3(\mathbb F_2)$, then they are contained in a maximal subalgebra of
$\M_3(\mathbb F_2)$. By Proposition~\ref{clasmax}, the maximal subalgebra is either a field or
it fixes a non-trivial proper subspace. Since $A^2=A$, the former case is not possible.
In the latter case, $A$ and $B$ have a common eigenvector either in their action on column vectors
or in their action on row vectors. It is however a straightforward verification to see that no such
common eigenvector exist. Thus $A$ and $B$ generate the algebra $\M_3(\mathbb F_2)$. In fact, in the same
way one can see that they generate $\M_3(\mathbb F_p)$ for any prime $p$ different from $3$.
With a bit more work, one can see that the subalgebra of $\M_3(\mathbb Z)$ generated by $A$ and $B$
has index 9. Note that by (\ref{n=3}), there are 129024 ordered pairs of $3\times 3$ matrices with
entries in $\{0,1\}$, which considered as elements of $\M_3(\mathbb F_2)$ generate the algebra
$\M_3(\mathbb F_2)$. Tsvetomira Radeva, at our request,  performed computations using Java and GAP
and found that among them exactly 9132 pairs do not generate $\M_3(\mathbb Z)$. The computations are based
on a result of Paz \cite{paz} and use the LLL algorithm \cite{LLL}, \cite{pohst}.}
\end{remark}

\vspace{4mm}
We end with the following curious observation. In Theorem~\ref{minnumbergen} we defined a family
of polynomials $\phi_k(x)$, $k\geq 2$. The polynomial $x^{3k-2}+\phi_k(x)$ is a factor
of the polynomial $f_k$ defined in Theorem~\ref{3by3}. Define polynomials $\psi_k(x)$
as follows.
\begin{equation}\psi_k(x)=\begin{cases}
\frac{x^{3k-2}+\phi_k(x)}{x-1}, & \text{if $k \equiv 0,\  4\,(\text{mod}\,6)$}\\
\frac{x^{3k-2}+\phi_k(x)}{x^2-1}, & \text{if $k \equiv 1,\ 3\,(\text{mod}\,6)$}\\
\frac{x^{3k-2}+\phi_k(x)}{x^3-1}, & \text{if $k \equiv 2\,(\text{mod}\,6)$}\\
\frac{x^{3k-2}+\phi_k(x)}{(x+1)(x^3-1)}, & \text{if $k \equiv 5\,(\text{mod}\,6)$.}
\end{cases}
\end{equation}
Computer computations with Maxima show that the polynomials $\phi_k$ and $\psi_k$ are irreducible
for $k\le 250$. While the polynomials $\phi_k$ have only six non-zero coefficients,
the polynomials $\psi_k$
have complicated structure. For example,
$\psi_{12}(x)={x}^{33}+{x}^{32}+{x}^{31}+{x}^{30}+{x}^{29}+{x}^{28}+{x}^{27}+{x}^{26}+{x}^{25}+{x}^{24}+{x}^{23}+{x}^{22}+2\,{x}^{21}+2\,{x}^{20}+2\,{x}^{19}+2\,{x}^{18}+2\,{x}^{17}+2\,{x}^{16}+2\,{x}^{15}+2\,{x}^{14}+2\,{x}^{13}+2\,{x}^{12}+{x}^{11}-{x}^{10}-2\,{x}^{9}-2\,{x}^{8}-2\,{x}^{7}-2\,{x}^{6}-2\,{x}^{5}-2\,{x}^{4}-2\,{x}^{3}-2\,{x}^{2}-2\,x-1$.
Nevertheless, it seems that all the coefficients of $\psi_k$ are in the set $\{-2,-1,0,1,2\}$.
Even though we do not have at present any conceptual reason for it, we propose the following
intriguing conjecture.

\begin{conjecture}
The polynomials $\phi_k$ and $\psi_k$ are irreducible.
\end{conjecture}

%Proposition \ref{2by2} allows to find explicitly two generators for $\M_2(\mathbb{Z})^{16}$.
%One such pair of generators is displayed in Table \ref{tablegen}. Each entry in the
%table displays the corresponding component of the two generators. Our first proof that
%$\M_2(\mathbb{Z})^{16}$ is $2$-generated was in fact by proving directly that this pair
%works. To do so, according to Lemma \ref{localgen}, Theorem \ref{simple}, and Proposition \ref{det},
%one only needs to verify that
%\begin{itemize}
%\item The commutator of each pair of matrices in Table \ref{tablegen} has determinant $\pm 1$.
%\item No two pairs of matrices in Table \ref{tablegen} are conjugate to each
%other modulo some prime.
%It suffices to check that no two pairs of matrices in Table \ref{tablegen} have the same
%$\conj$ modulo some prime.
%\end{itemize}
%\noindent
%This verification is straightforward.

\end{document}